\documentclass[11pt]{amsart}

\usepackage{enumerate,color,xcolor}

\newtheorem{lem}{Lemma}[section]

\newtheorem{thm}{Theorem}[section]
\newtheorem{cor}{Corollary}[section]

\newtheorem{rmk}{Remark}[section]

\numberwithin{equation}{section}

\def\<{\langle}
\def\>{\rangle}
 
\title [LDP for the mean-field limit of Hawkes processes]
{Large deviations for the mean-field limit of Hawkes processes}

\author{Fuqing GAO}
\address{School of Mathematics and Statistics, Wuhan University, Wuhan 430072, China}
\email{fqgao@whu.edu.cn}

\author{Lingjiong Zhu}
\address{Department of Mathematics, Florida State University, 1017 Academic Way, Tallahassee, FL-32306, United States of America}
\email{zhu@math.fsu.edu}

\date{May 19, 2026}

\begin{document}

\subjclass[2010]{60G55, 82C22, 60F10}

\keywords{Hawkes process, mean-field limit, large deviations.}

\begin{abstract}
Hawkes processes are a class of simple point processes
whose intensity depends on the past history, and is in general non-Markovian.
Limit theorems for Hawkes processes in various asymptotic regimes have been studied in the literature.
In this paper, we study a multidimensional nonlinear Hawkes process
in the asymptotic regime when the dimension goes to infinity,
whose mean-field limit is a time-inhomogeneous Poisson process, 
and our main result is a large deviation principle for the mean-field limit.
\end{abstract}

\maketitle

\section{Introduction and main results}

\subsection{Introduction}

Among the multivariate counting processes, mutually exciting processes,
also known as the Hawkes process \cite{Hawkes}, are one of the most popular
models to describe the interactions across the dimensions and also the dependence
on the past events. For a multivariate Hawkes process, a jump in one component will not only increase
the intensity of future jumps of its own component, known as the self-exciting property,
but also increase the intensity of the future jumps of or the other components
that are connected to its own component, which is known as the mutually-exciting property (see \cite{Hawkes}).

Let us now formally define the multivariate Hawkes process.
Let $N$ be a simple point process on $\mathbb{R}_{+}$, that is, 
a family $\{N(C)\}_{C\in\mathcal{B}(\mathbb{R}_{+})}$ of random variables
with values in $\mathbb{N}\cup\{\infty\}$ indexed
by the Borel $\sigma$-algebra $\mathcal{B}(\mathbb{R}_{+})$ of the positive real line $\mathbb{R}_{+}$,
where $N(C)=\sum_{n=1}^{\infty}1_{T_{n}\in C}$ and $(T_{n})_{n=1}^{\infty}$
is a sequence of extended positive real-valued random variables so that almost surely $0<T_{1}$, $T_{n}<T_{n+1}$
on $\{T_{n}<\infty\}$ for every $n=1,2,3\ldots$. 
Let $\mathcal{F}_{t}=\sigma(N(C),C\in\mathcal{B}(\mathbb{R}_{+}),C\subset(0,t])$, where $t>0$. 
The process $\lambda_{t}$ is called the $\mathcal{F}_{t}$-intensity of $N$ if for
all intervals $(a,b]$, where $0<a<b<\infty$, we have
\begin{equation}
\mathbb{E}[N((a,b])|\mathcal{F}_{a}]
=\mathbb{E}\left[\int_{a}^{b}\lambda_{s}ds\Big|\mathcal{F}_{a}\right],
\qquad
\text{a.s.}
\end{equation}
The $N$-dimensional Hawkes process is defined as $\left(Z_{t}^{1},\ldots,Z_{t}^{N}\right)$,
where $Z_{t}^{i}$, $1\leq i\leq N$, are simple point processes without common points,
with $Z_{t}^{i}$ admitting an $\mathcal{F}_{t}$-intensity:
\begin{equation}
\lambda_{t}^{i}:=\phi_{i}\left(\sum_{j=1}^{N}\int_{0}^{t-}h_{ij}(t-s)dZ_{s}^{j}\right),\label{dynamics}
\end{equation}
where $\phi_{i}(\cdot):\mathbb{R}^{+}\rightarrow\mathbb{R}^{+}$ is locally integrable, left continuous,
$h_{ij}(\cdot):\mathbb{R}^{+}\rightarrow\mathbb{R}^{+}$ and
we always assume that $\Vert h_{ij}\Vert_{L^{1}}=\int_{0}^{\infty}h_{ij}(t)dt<\infty$.
In \eqref{dynamics}, $\int_{0}^{t-}h_{ij}(t-s)dZ_{s}^{j}$ stands for $\sum_{0<\tau^{j}<t}h_{ij}(t-\tau^{j})$, where
$\tau^{j}$ are the occurrences of the points before time $t$ of the counting process $Z^{j}$. 
In the literature, $h_{ij}(\cdot)$ and $\phi_{i}(\cdot)$ are usually referred to
as exciting function (or sometimes kernel function) and rate function respectively.
A Hawkes process is linear if $\phi_{i}(\cdot)$ are all linear and 
it is nonlinear otherwise.

The Hawkes process when $\phi_{i}(\cdot)$ are linear was first proposed by Alan Hawkes
in 1971 to model earthquakes and their aftershocks \cite{Hawkes}. 
The nonlinear Hawkes process was first introduced by Br\'{e}maud and Massouli\'{e} \cite{Bremaud}.
The Hawkes process naturally generalizes the Poisson process and it
captures both the self-exciting and mutually-exciting property and the clustering effect, 
and it is a very versatile model
for statistical analysis. These explain why it has wide applications in neuroscience, genome analysis,
criminology, social networks, healthcare, seismology, insurance, finance, machine learning and many other fields.
For a list of references, we refer to \cite{ZhuThesis}.

The linear Hawkes process has the immigration-birth representation, see e.g. \cite{HawkesII}, that makes
it more analytically tractable than the nonlinear Hawkes process. 
Recently, renewal properties have also been investigated for Hawkes processes, see e.g. \cite{Costa}
for the nonlinear Hawkes process with bounded memory and \cite{Graham} for the linear Hawkes process
with unbounded memory.

Most of the asymptotic results for Hawkes processes in the literature
are for the univariate Hawkes process, that is the one-dimensional case, i.e.
$N=1$ in \eqref{dynamics}.
Among these asymptotic results, 
the large time limit theorems are the most studied. 
For the linear Hawkes process,
the functional law of large numbers
and functional central limit theorems were studied in Bacry et al. \cite{Bacry};
the large deviations principle was studied in Bordenave
and Torrisi \cite{Bordenave}.
The precise large and moderate deviations are recently studied in Gao and Zhu \cite{GaoFZhu}.
Limit theorems for marked Hawkes processes were studied
for the univariate case in Karabash and Zhu \cite{KarabashZhu2015}. 
For multivariate marked Hawkes processes (i.e. general $N\in\mathbb{N}$), 
functional law of large numbers and central limit theorem were obtained in
Horst and Xu \cite{HorstXu2021}. Karim et al. \cite{Karim2025} established a large deviation principle 
for a multivariate compound process induced by a multivariate marked Hawkes process.
Blanchet et al. \cite{Blanchet2025} studied sample path large deviations for multivariate Hawkes processes when the mutual excitation rates are heavy-tailed.
Recently, Horst and Xu \cite{HorstXu2025} established a nearly full functional law of large numbers and central limit theorem.

For the nonlinear Hawkes process, Zhu \cite{ZhuCLT} studied the functional central limit theorems
by using Poisson embeddings and a careful analysis of the decay of the correlations over time. 
In \cite{ZhuII}, Zhu obtained a process-level, i.e. level-3 large deviation principle
and the rate function is expressed as a variational problem optimizing over
a certain entropy function of any simple point process against the underlying nonlinear Hawkes process.
When the exciting function is exponential and the process is Markovian, 
an alternative expression for the rate function
for the large deviations was obtained in Zhu \cite{ZhuI}.
Using the techniques as a combination of Poisson embeddings, Stein's method and Malliavin calculus, 
the quantitative Gaussian and Poisson approximations were studied in Torrisi \cite{TorrisiI,TorrisiII}.
Functional inequalities are studied in \cite{Flint}.

There have been some progress made in the direction of asymptotic results
other than the large time limits for Hawkes process. For instance, 
in the case of linear Hawkes process, the scaling limit theorems for nearly unstable, also known as, nearly critical case,  are studied
in Jaisson and Rosenbaum \cite{Jaisson,JaissonII}.
Horst et al. \cite{Horst2023} show that the rescaled intensities converges weakly to the rough fractional diffusion
and analogous results for the multivariate case
are obtained in \cite{Euch2018,Rosenbaum2021}.
Scaling limit theorems have also been obtained for marked Hawkes processes for both the univariate case \cite{HorstXu2022} and the multivariate case \cite{Xu2024,Xu2024arXiv}.
Xu \cite{XuVolterraI,XuVolterraII} obtained scaling limit theorems for a class of heavy-tailed marked Hawkes processes 
and proved that their rescaled intensities converge weakly to the unique solution of a stochastic Volterra equation.
When the exciting function is exponential,
the intensity process and the pair of the counting process and the intensity process are Markovian. 
In Gao and Zhu \cite{GZ}, they studied the functional central limit theorems
for the linear Hawkes process
when the initial intensity is large, and they further studied the large deviations
and applied their results to insurance and queueing systems in \cite{GZ2}.
For the more general linear and non-Markovian case, Gao and Zhu \cite{GZ3} considered the large
baseline intensity asymptotic results in the stationary regime and studied the applications to queueing systems;
Li and Pang \cite{LiPang2022SPA} obtained the large intensity asymptotic results for marked Hawkes processes in the non-stationary regime.

The first work on the mean-field limit for high dimensional Hawkes processes \eqref{dynamics} appeared in Delattre et al. \cite{Delattre}. 
They showed that under a certain setting, the mean-field limit is an inhomogeneous Poisson process.
Since their seminal work, mean-field limit for Hawkes processes and the generalizations
have attracted a lot of attention in the recent literature. 
Chevallier \cite{Chevallier} studied a generalized Hawkes process model with an inclusion of the dependence on
the age of the process. Raad et al. \cite{Raad} also studied mean-field limit for an age dependent Hawkes process.
Chevallier \cite{ChevallierIII} used mean-field limit for an age-dependent Hawkes process
and a notation of stimulus sensitivity to study the response of the network to a stimulus. 
Delattre and Fournier \cite{DF} studied the mean-field limit
for Hawkes processes on a graph with two nodes whether or not influence each other modeled
by i.i.d. Bernoulli random variables. Dittlevsen and L\"{o}cherbach \cite{DL} considered
a multi-class systems of interacting nonlinear Hawkes processes modeling several
large families of neurons and studied the associated mean-field limit.
Based on the model in \cite{DL}, Chevallier et al. \cite{Chevallier2021} obtained a strong error bound and a diffusion approximation, moment bounds for the resulting diffusion, and numerical schemes.
Chevallier et al. \cite{CDLO} studied the spatially extended systems of interacting nonlinear Hawkes processes modeling large systems of neurons and study the associated mean field limit, 
which can be described by a neural field equation.
Recently, Lu\c{c}on and Poquet \cite{Lucon2025} studied the long-term stability of interacting Hawkes processes
with spatial extension following the work of \cite{CDLO} concerning the approximation of neural field equations by Hawkes processes.
Erny et al. \cite{Erny} studied the mean-field limit for interacting Hawkes
processes in a diffusive regime. 
L\"{o}cherbach \cite{Locherbach} studied Freidlin-Wentzell type large deviations
for cascades of diffusions that arise from an oscillating system of interacting Hawkes processes.
Heesen and Stannat \cite{Heesen} studied asymptotic fluctuations
of mean-field interacting non-linear Hawkes processes 
by obtaining a function central limit theorem that characterizes the asymptotic fluctuations
in terms of a stochastic Volterra integral equation.
Pfaffelhuber et al. \cite{Pfaffelhuber} studied mean-field limit for Hawkes processes
with excitation and inhibition. 
Stiefel \cite{Stiefel2023} studied mean-field limit for Hawkes processes with inhibition on Erd\H{o}s-R\'{e}nyi graphs.
Duval et al. \cite{Duval} studied a general class of mean-field interacting nonlinear Hawkes processes
with multiplicative inhibition.
Agathe-Nerine \cite{Agathe} studied the mean-field limit for Hawkes processes on inhomogeneous random graphs.
An age and leaky memory dependent Hawkes process was proposed in 
\cite{Schmutz} and the mean-field limit was obtained.
Szymanski and Xu \cite{Szymanski-Xu-2025} obtained mean-field limits for Hawkes processes
in the nearly unstable regime by extending the methodologies developed in \cite{Jaisson}, 
which generalized the results of \cite{Delattre}.
In terms of the statistical theory of mean-field Hawkes processes, 
Bacry et al. \cite{Bacry2016} studied the problem of
estimating the underlying parameters in systems of mean field interacting Hawkes processes,
and Delattre and Fournier \cite{DF} estimated parameters of Hawkes processes an an Erd\H{o}s-R\'{e}nyi graph. 
Very recently, Duarte et al. \cite{Duarte2025} studied the nonparametric estimation
based on a Nadaraya-Watson type kernel estimator.
Borovykh et al. \cite{Borovykh} applied the mean-field limit for Hawkes processes
to study the systemic risk in a mean-field model of interbank lending with self-exciting shocks.
L\"{o}cherbach \cite{Locherbach2017} presented a short survey of some aspects of the study of Hawkes processes in high dimensions including the mean-field limit, 
in view of modeling large systems of interacting neurons.

Despite a significant number of studies
about the mean-field limit for Hawkes processes and the variants in the recent literature, 
surprisingly, there has not been any study about the large deviations of the mean-field limit,
which concerns the small probability of rare events that the mean-field of the Hawkes process
deviates away from its typical limit.
To the best of our knowledge,  
our paper is the first one to study
the large deviations for the mean-field limit of a multidimensional nonlinear Hawkes process \eqref{dynamics}, 
which fills in this gap in the literature.
We follow the same model setup as the first mean-field limit paper in this paper, i.e. \cite{Delattre},
although our methods and techniques might be extendable to study the large deviations
for the mean-field limit for the various generalizations of the Hawkes model, 
which will be left as future research directions.

The paper is organized as follows. In Section~\ref{sec:pre}, we provide
the preliminaries that will be used to state our main results and the proofs.
In Section~\ref{sec:main:results}, we state the main result of the paper, which
is a large deviation principle for the mean-field limit. 
The proofs are then presented in the subsequent sections. 
In particular, we show existence, uniqueness and a law of large numbers result for perturbed mean fields
of Hawkes processes in Section~\ref{sec:LLN}. 
We analyze the rate function in Section~\ref{sec:rate:function}, and provide
a precise representation and an approximation theorem.
Finally, we are ready to provide the proof of the main result
in Section~\ref{LDPMeanFieldsSection}, in particular, the upper bound for the large deviations in Section~\ref{sec:upper:bound}
and the lower bound in Section~\ref{sec:lower:bound}. Further discussions are provided in Section~\ref{sec:discussions}. 

%%%%%%%%%%%%%%%%%%%%%%%%%%%%%%%%%%%%%%%%%
\subsection{Preliminaries}\label{sec:pre}

In order to state our main results, let us first provide some preliminaries.
First, let us introduce the assumptions that will be used throughout
the rest of the paper.
%%%%%%%%%%%%%%%%%%%%%%%%%%%%%%%%%%%%%%%%%%%%%%%%%%%%
 
For each $N\geq 1$, we let $h_{ij}=\frac{1}{N}h$
and $\phi_{i}=\phi$ in \eqref{dynamics} throughout the rest of the paper,
which is the setting for the mean-field limit of Hawkes processes (see e.g. \cite{Delattre}).

Let $\phi:\mathbb R \mapsto [0,\infty)$ be such that
$\|\phi\|_{lip}=\sup_{x\ne y}|x-y|^{-1}|\phi(x)-\phi(y)|<\infty$
and let $h:[0,\infty)\mapsto \mathbb R$ be a locally square integrable function.
Moreover, we assume the following.

\begin{itemize}
\item[(A.1).] 
$h(\cdot):[0,\infty)\rightarrow[0,\infty)$ is locally integrable
and locally bounded; $h$ is differentiable and $|h'|$ is locally integrable and locally bounded.
 
\item[(A.2).]
$\phi(\cdot):[0,\infty)\rightarrow(0,\infty)$ 
is $\alpha$-Lipschitz for some $0<\alpha<\infty$ and $\alpha\Vert h\Vert_{L^{1}[0,T]}<1$.

\item[(A.3).]
$\inf_{x\geq 0}\phi(x)>0$.
\end{itemize}

%%%%%%%%%%%%%%%%%%%%%%%%%%%%%%%%%%% 
Next, let us introduce the technical backgrounds and notations.

\subsubsection*{Wasserstein metric}

Let $M_1(\mathbb N)$ be the space of probability measures on $\mathbb N=\{0,1,2,\cdots\}$ with
$\int_{\mathbb N} x\mu(dx)<\infty$.
Let $W_1$ be the $L^1$-Wasserstein metric on $M_1(\mathbb N)$   defined by
\begin{equation*}
W_1(\nu,\mu)=\inf_{\pi} \int_{\mathbb N\times \mathbb N} |x-y|d\pi(x,y),\quad \nu,\mu\in M_1(\mathbb N),
\end{equation*}
where the infimum above is taken over all probability measures $\pi$ on
the product space $\mathbb N\times \mathbb N$ with  marginal distributions being $\nu$
and $\mu$ respectively.
According to the Monge-Kantorovitch-Rubinstein dual representation formula, 
\begin{equation*}
W_1(\nu,\mu)=\sup \left\{ \int_{\mathbb N} \varphi (x) (\nu-\mu)(dx);~\|\varphi\|_{lip}\leq 1\right\}.
\end{equation*}

It is known that for any $\mu_n,\mu\in (M_1(\mathbb N),W_1)$, convergence in $W_{1}$ is equivalent to weak convergence
plus the convergence of the first moment, i.e. $\mu_n\stackrel{W_1}{\to} \mu$ if and only if
\begin{equation*}
\mu_n\stackrel{w}{\to} \mu \mbox{ and } \int_0^\infty x\mu_n(dx)\to \int_0^\infty x\mu(dx).
\end{equation*}
For the background and properties of Wasserstein metric, we refer to \cite{Villani2003}.

%%%%%%%%%%%%%%%%%%%%%%%%%%%%%%%%
\subsubsection*{The path space} 
Let  $D\left([0,T],M_{1}(\mathbb N)\right)$ denote the space of $M_1(\mathbb N)$-valued c\`{a}dl\`{a}g functions on $[0,T]$ equipped with the topology of uniform convergence. 

For any $\mu\in D\left([0,T],M_1(\mathbb N)\right)$,  
set $\bar{\mu}_t:=\int_{\mathbb N} x \mu_t(dx)$ 
and let $F_\mu$ and $f_\mu$ denote its distribution function and density function respectively, i.e.,
\begin{equation*}
F_\mu(t,x):=\mu_t([0,x]),\mbox{ \quad } f_\mu(t,x):=\mu_t(\{x\}).
\end{equation*}
Set $\bar{F}_\mu(t,x):=1-\mu_t([0,x])$.     
Let us define
\begin{equation*}
M[0,T]:=\left\{\mu\in D\left([0,T],M_1(\mathbb N)\right);~\mu_{0}(\{0\})=1, F_\mu(t,x) \mbox{ is decreasing in } t \right\}.
\end{equation*}
Then for any $\mu\in M[0,T]$,  $t\mapsto \bar{\mu}_t$ is increasing. 

Let $\langle \phi, \zeta\rangle$ denote the integral of the function $\phi$ with respect to a finite signed measure $\zeta$ on $\mathbb R $.
For any $\psi(t,x)$ bounded in $(t,x)$ and Lipschitz in $x$, 
$\mu\in D\left([0,T],M_1(\mathbb N)\right)$, $ t \in[0,T]$, we abbreviate
\begin{equation*}
\langle \mu_t, \psi(t)\rangle :=\int_{\mathbb N}\psi(t,x)\mu_t(dx).
\end{equation*}

%%%%%%%%%%%%%%%%%%%%%%%%%
\subsubsection*{The mean-field limit}

By using the Poisson embeddings, see e.g \cite{Bremaud,Delattre},
we can express the Hawkes process $\left(Z_{t}^{1},\ldots,Z_{t}^{N}\right)$
as the solution of a Poisson driven SDE:
\begin{equation}
Z^{i}_t=\int_0^t \int_0^\infty I_{\left\{z \leq \phi_{i}\left(\sum_{j=1}^N\int_0^{s-}h_{ij}(s-u)dZ_u^{j}\right)\right\}}
\pi^i(ds\,dz),
\qquad
1\leq i\leq N,
\end{equation}
where $\{\pi^i(ds\,dz), i\geq 1\}$ are a sequence of  i.i.d.  Poisson
measures with common intensity measure $dsdz$ on
$[0,\infty) \times [0,\infty)$. 

For each $N\geq 1$, let $h_{ij}=\frac{1}{N}h$
and $\phi_{i}=\phi$ and
we consider the Hawkes process
$\left(Z^{N,1}_t,\cdots,Z^{N,N}_t\right)_{t\geq 0}$, where
\begin{equation} \label{N-dim-Hawkes-process-eq}
Z^{N,i}_t=\int_0^t \int_0^\infty I_{\left\{z \leq \phi \left(N^{-1}\sum_{j=1}^N\int_0^{s-}h(s-u)dZ_u^{N,j}\right)\right\}}
\pi^i(ds\,dz).
\end{equation}

Let $m_t $ be the unique non-decreasing locally bounded solution of the equation, see e.g. Delattre et al. \cite{Delattre}:
\begin{equation} \label{mean-mean-eq}
m_t = \int_0^t \phi\left( \int_0^s h(s-u)dm_u\right)ds, \;\;\text{ }\;\;t\geq 0.
\end{equation}

Consider the equation:
\begin{equation}\label{Hawkes-mean-eq}
\tilde Z_t = \int_0^t\int_0^\infty I_{ \left\{z \leq \phi\left(\int_0^s 
h(s-u)d \mathbb E(\tilde Z_u)\right)\right\}}\pi(ds\,dz),\;\;\text{ }\;\;t\geq 0,
\end{equation}
where $\pi(ds \, dz)$ is a Poisson measure on $[0,\infty)\times [0,\infty)$ with intensity measure $ds dz$.
Note that as observed in \cite{Delattre}, $\tilde{Z}_{t}$ given in \eqref{Hawkes-mean-eq} 
is an inhomogeneous Poisson process with intensity $m_{t}$ given in \eqref{mean-mean-eq}.
 
The \textit{mean process} of the Hawkes processes is defined by
$\left(Z^{N,1}_t,\cdots,Z^{N,N}_t\right)_{t\geq 0}$:
\begin{equation} \label{Hawkes-mean-process-eq-0}
\overline{Z}^{N}_t=\frac{1}{N}\sum_{i=1}^N Z^{N,i}_t,~t\geq 0.
\end{equation}
Limit theorems including fluctuations, large and moderate deviations for the mean process have been studied in \cite{GaoFZhu-b}.

The \textit{mean-field} of the Hawkes processes is defined as 
the stochastic process taking values in $ M[0,T]$:
\begin{equation}\label{empirical-measure-def-eq}
L^N(t,dx)=\frac{1}{N}\sum_{i=1}^N \delta_{Z^{N,i}_t}(dx),~~0\leq t\leq T.
\end{equation}
Under this setup, Delattre et al. \cite{Delattre} showed that
$L^N$ converges to its mean-field limit:
\begin{equation*}
L^N\to \mathcal L \mbox { in } M[0,T],
\end{equation*}
where  $\mathcal L_t(dx)$  is the distribution of $\tilde Z_t$.

\subsubsection*{Large deviation principle}
In probability theory, a sequence $(P_{n})_{n\in\mathbb{N}}$ of probability measures on a topological space $X$ 
is said to satisfy the large deviation principle (LDP) with rate function $I:X\rightarrow\mathbb{R}\cup\{\infty\}$ if $I$ is non-negative, 
lower semicontinuous and for any measurable set $A$, with $A^{o}$ denoting ithe interior of $A$ and $\overline{A}$ being its closure, 
\begin{equation}\label{eqn:LDP}
-\inf_{x\in A^{o}}I(x)\leq\liminf_{n\rightarrow\infty}\frac{1}{n}\log P_{n}(A)
\leq\limsup_{n\rightarrow\infty}\frac{1}{n}\log P_{n}(A)\leq-\inf_{x\in\overline{A}}I(x).
\end{equation}
Note that \eqref{eqn:LDP} is equivalent to 
$\limsup_{n\rightarrow\infty}\frac{1}{n}\log P_{n}(F)\leq-\inf_{x\in F}I(x)$
for any closed subset $F$ and $\liminf_{n\rightarrow\infty}\frac{1}{n}\log P_{n}(O)\geq-\inf_{x\in O}I(x)$
for any open subset $O$.
We refer to Dembo and Zeitouni \cite{Dembo} 
and Varadhan \cite{VaradhanII} 
for general background regarding the theory of large deviations and their applications.

%%%%%%%%%%%%%%%%%%%%%%%%%%%%%%%%%%%%%
\subsection{Main results}\label{sec:main:results}

In this section, we state the main result of our paper, which is a large deviation principle
for the mean-field of Hawkes processes. 

\begin{thm}\label{main-LDP-thm-1}  Suppose (A.1), (A.2)  and   (A.3) hold.  For any $\mu\in  M[0,T]$, define
\begin{equation}\label{main-LDP-thm-1-rate-function}
\begin{aligned}
I(\mu) =&
\left\{
\begin{array}{l}\displaystyle\int_0^T\int_{0}^\infty \bigg( {G_\mu(t,x)} \log {G_\mu(t,x)} -{G_\mu(t,x)}+1\bigg)  \phi\left(\int_0^{t} h(t-s)d\bar{\mu}_s\right)\mu_t(dx)dt,\\
\quad\quad\quad\quad\quad\quad\quad\quad\quad \mbox{ if } \partial_tF_\mu(t,x)dt\ll f_\mu(t,x)dt \mbox{ for each } x\in\mathbb N,\\
\hbox{}\\
+\infty, \quad\quad\quad\quad\quad\quad\quad\quad\quad\mbox{otherwise},
\end{array}
\right.
\end{aligned}
\end{equation}
where
\begin{equation*}
G_\mu(t,x)=\frac{-\partial_tF_\mu(t,x)}{f_\mu(t,x)\phi\left( \int_0^th(t-u)d\bar{\mu}_u\right)}.
\end{equation*} 
Then for any closed subset $F\subset M[0,T]$,
\begin{equation}\label{main-LDP-thm-1-ub}
\limsup_{N\to \infty}\frac{1}{N}\log \mathbb P\left(L^N \in F\right)\leq -\inf_{\mu\in F}I(\mu),
\end{equation}
and for any open subset $O\subset M[0,T]$, 
\begin{equation} \label{main-LDP-thm-1-lb}
\liminf_{N\to\infty}\frac{1}{N}\log \mathbb P \left(L^N\in O\right)
\geq-\inf_{\mu\in O}I(\mu).
\end{equation}
\end{thm}

The proof of Theorem~\ref{main-LDP-thm-1} will be provided
in Section~\ref{LDPMeanFieldsSection}. 
The proof of the upper bound \eqref{main-LDP-thm-1-ub}
will be provided in Section~\ref{sec:upper:bound},
which relies on establishing the exponential tightness. 
The proof of the lower bound \eqref{main-LDP-thm-1-lb}
will be provided in Section~\ref{sec:lower:bound},
and the main idea is to optimize over tilted probability measures
such that under a tilted probability measure, the mean-field
model is perturbed, and the existence, uniqueness and a law of large numbers result for the perturbed mean fields
will be provided in the following Section~\ref{sec:LLN}.
Finally, we provide the 
analysis of the rate function $I(\mu)$ in Section~\ref{sec:rate:function}
that will be used to match the upper and lower bounds.

%%%%%%%%%%%%%%%%%%%%%%%%%%%%%%%%%%%%%%%%%%%%%%%%%%%%%%%%%%%%%%%%%%%%%%%%%%%%%
\section{Law of large numbers for perturbed mean fields}\label{sec:LLN}

In this section, we obtain existence, uniqueness and a law of large numbers result for perturbed mean fields. 
First of all, we set
\begin{align*}
&B([0,T]\times\mathbb N):=\left\{
\varphi: [0,T]\times \mathbb N\to [-\infty,+\infty)\quad\text{is measurable}\right\}.
\end{align*}

For any measurable function  $\psi: [0,T]\times \mathbb N\to \mathbb R$,    we define  $\nabla \psi(s,x):=\psi(s,x)-\psi(s,x-1)$ for any $ x=1,2,\ldots$, and $\nabla \psi(s,0):= \psi(s,0)$, and for any $s\in[0,T]$,   denote by  $\nabla\psi(s): \mathbb N\ni x\to \nabla \psi(s,x)$.
We also define the following two spaces:
$$
 C^{1,b}([0,T]\times\mathbb N ):=\left\{\psi\text{ is  continuous  differentiable in $t$}, \sup_{t\in [0,T]}\sup_{x\in \mathbb N}|\psi(t,x)|<\infty\right \},
$$
and
$$
C^{1,lip}([0,T]\times\mathbb N):=\left\{\psi\text{ is  continuous  differentiable in $t$}, \sup_{t\in [0,T]}\sup_{x\in \mathbb N}|\nabla \psi(t,x)|<\infty\right \}.
$$

For any $\varphi\in B([0,T]\times\mathbb N)$,
we consider the perturbed SDEs, 
$\left(Z^{\varphi,N,1}_t,\cdots,Z^{\varphi,N,N}_t\right)_{t\geq 0}$, defined as:
\begin{equation}\label{N-dim-perturbation-Hawkes-process-eq}
Z^{\varphi,N,i}_t
=\int_0^t \int_0^\infty  
I_{\left\{z \leq e^{\varphi\left(s,Z^{\varphi,N,i}_{s-}\right)}\phi \left( \int_0^{s-}h(s-u)d\overline{Z}_u^{\varphi,N}\right)\right\}}
\pi^i(ds\,dz),~i=1,\cdots, N,
\end{equation}
where
\begin{equation} \label{N-dim-perturbation-Hawkes-process-mean-eq}
\overline{Z}^{\varphi,N}_t=\frac{1}{N}\sum_{i=1}^N Z^{\varphi,N,i}_t,~t\geq 0.
\end{equation}

First, we show that the perturbed mean-field model \eqref{N-dim-perturbation-Hawkes-process-eq} has a unique strong solution
with finite second moments.

\begin{lem}\label{N-dim-perturbation-Hawkes-process-bounded-Thm}
Suppose (A.1), (A.2)  and   (A.3) hold.  If we further assume that
\begin{equation} \label{N-dim-perturbation-Hawkes-process-varphi-bounded-c}
\sup_{t\in [0,T]}\sup_{x\in \mathbb N}{|\varphi(t,x)|}<\infty,
\end{equation}
then, the equation \eqref{N-dim-perturbation-Hawkes-process-eq}  has a unique strong solution   $\left(Z^{\varphi,N,1}_t,\cdots,Z^{\varphi,N,N}_t\right)_{t\in [0,T]}$ in $D\left([0,T], \mathbb N^N\right)$.  Moreover, the solution satisfies  
\begin{equation}\label{uniform:L2:finite}
\sum_{1\leq i \leq N}  \mathbb E\left(\left|Z_t^{\varphi,N,i}\right|^2\right)<\infty,~t\geq 0.
\end{equation}
\end{lem}

\begin{proof} 
Let us first prove that 
the equation \eqref{N-dim-perturbation-Hawkes-process-eq} has a unique strong solution.  For each $L\geq 1$, we consider the perturbed SDEs with truncated rate  $\phi_{L}(x)$:
\begin{equation}\label{N-dim-perturbation-truncated-Hawkes-process-eq}
Z^{L,\varphi,N,i}_t
=\int_0^t \int_0^\infty  
I_{\left\{z \leq e^{\varphi\left(s,Z^{L,\varphi,N,i}_{s-}\right)}\phi_L \left( \int_0^{s-}h(s-u)d\overline{Z}_u^{L,\varphi,N}\right)\right\}}
\pi^i(ds\,dz),~i=1,\ldots, N,
\end{equation}
where
\begin{equation} \label{N-dim-perturbation-truncated-Hawkes-process-mean-eq}
\phi_{L}(x):=\phi(x\wedge L),~\quad
\overline{Z}^{L,\varphi,N}_t:=\frac{1}{N}\sum_{i=1}^N Z^{L,\varphi,N,i}_t,~t\geq 0.
\end{equation}

Now, let us prove the existence of the equation  \eqref{N-dim-perturbation-truncated-Hawkes-process-eq}  by a Picard iteration argument. 
Let $Z^{i,0}_t := 0$ and, for $n\geq 1$,
\begin{equation}
Z_t^{i,n+1}:= \int_0^t \int_0^\infty  I_{\left\{z \leq e^{\varphi\left(s,Z_{s-}^{i,n}\right)}\phi_{L}\left( \int_0^{s-}h(s-u)d\overline{Z}_u^{n}\right)\right\}}\pi^i(ds\,dz)\,,
\end{equation}
where $\overline{Z}_t^{n}:=\frac{1}{N} \sum_{i=1}^N Z_t^{i,n}$.

Next, we define $\delta^{n}_t:=\mathbb E\left(\sup_{0\leq s\leq t} \left|Z_s^{i,n+1}-Z_s^{i,n}\right|^2\right)$.
Then, for $n\geq 0$,
\begin{equation*}\label{N-dim-perturbation-Hawkes-process-bounded-Thm-eq-2}
\begin{aligned}
\delta^{n}_t
&=
\mathbb{E}\bigg[\sup_{0\leq s\leq t}
\bigg|\int_0^s \int_0^\infty  I_{\left\{z \leq e^{\varphi\left(v,Z_{v-}^{i,n}\right)}\phi_{L} \left( \int_0^{v-}h(v-u)d\overline{Z}_u^{n}\right)\right\}}\pi^i(dv\,dz)
\\
&\qquad\qquad
-\int_0^s \int_0^\infty  I_{\left\{z \leq e^{\varphi\left(v,Z_{v-}^{i,n-1}\right)}\phi_{L} \left( \int_0^{v-}h(v-u)d\overline{Z}_u^{n-1}\right)\right\}}\pi^i(dv\,dz)\bigg|^{2}\bigg]
\\
&\leq
2\mathbb{E}\bigg[\sup_{0\leq s\leq t}
\bigg|\int_0^s \int_0^\infty  I_{\left\{z \leq e^{\varphi\left(v,Z_{v-}^{i,n}\right)}\phi_{L} \left( \int_0^{v-}h(v-u)d\overline{Z}_u^{n}\right)\right\}}(\pi^i(dv\,dz)-dvdz)
\\
&\qquad\qquad
-\int_0^s \int_0^\infty  I_{\left\{z \leq e^{\varphi\left(v,Z_{v-}^{i,n-1}\right)}\phi_{L} \left( \int_0^{v-}h(v-u)d\overline{Z}_u^{n-1}\right)\right\}}(\pi^i(dv\,dz)-dvdz)\bigg|^{2}\bigg]
\\
&\qquad
+2\mathbb{E}\bigg[\sup_{0\leq s\leq t}
\bigg|\int_0^s \int_0^\infty  I_{\left\{z \leq e^{\varphi\left(v,Z_{v-}^{i,n}\right)}\phi_{L} \left( \int_0^{v-}h(v-u)d\overline{Z}_u^{n}\right)\right\}}dvdz
\\
&\qquad\qquad
-\int_0^s \int_0^\infty  I_{\left\{z \leq e^{\varphi\left(v,Z_{v-}^{i,n-1}\right)}\phi_{L} \left( \int_0^{v-}h(v-u)d\overline{Z}_u^{n-1}\right)\right\}}dvdz\bigg|^{2}\bigg],
\end{aligned}
\end{equation*}
and furthermore, we have
\begin{equation*}
\begin{aligned}
\delta^{n}_t
&\leq
8\mathbb{E}\bigg[
\bigg|\int_0^t \int_0^\infty  I_{\left\{z \leq e^{\varphi\left(s,Z_{s-}^{i,n}\right)}\phi_{L} \left( \int_0^{s-}h(v-u)d\overline{Z}_u^{n}\right)\right\}}(\pi^i(ds\,dz)-dsdz)
\\
&\qquad\qquad
-\int_0^t \int_0^\infty  I_{\left\{z \leq e^{\varphi\left(s,Z_{s-}^{i,n-1}\right)}\phi_{L} \left( \int_0^{s-}h(s-u)d\overline{Z}_u^{n-1}\right)\right\}}(\pi^i(ds\,dz)-dsdz)\bigg|^{2}\bigg]
\\
&\qquad
+2\mathbb{E}\bigg[
\bigg(\int_0^t \bigg|\int_0^\infty  I_{\left\{z \leq e^{\varphi\left(s,Z_{s-}^{i,n}\right)}\phi_{L} \left( \int_0^{s-}h(s-u)d\overline{Z}_u^{n}\right)\right\}}
\\
&\qquad\qquad\qquad\qquad
-I_{\left\{z \leq e^{\varphi\left(s,Z_{s-}^{i,n-1}\right)}\phi_{L} \left( \int_0^{s-}h(s-u)d\overline{Z}_u^{n-1}\right)\right\}}dz\bigg|ds\bigg)^{2}\bigg]
\\
&\leq 
8\mathbb E\bigg( \int_0^t  \int_0^\infty \bigg|I_{\left\{z \leq e^{\varphi\left(s,Z_{s-}^{i,n}\right)}\phi_{L} \left( \int_0^{s-}h(s-u)d\overline{Z}_u^{n}\right)\right\}}
\\
&\qquad\qquad\qquad\qquad
-I_{\left\{z \leq e^{\varphi\left(s,Z_{s-}^{i,n-1}\right)}\phi_{L} \left( \int_0^{s-}h(s-u)d\overline{Z}_u^{n-1}\right)\right\}}\bigg|^2 dz ds\bigg)\\
&
\qquad\qquad
+2T \mathbb E\bigg(\int_0^t \bigg| e^{\varphi\left(s,Z_{s-}^{i,n}\right)}\phi_{L} \left( \int_0^{s-}h(s-u)d\overline{Z}_u^{n}\right) \\
&
\qquad\qquad\qquad\qquad\qquad
- e^{\varphi\left(s,Z_{s-}^{i,n-1}\right)}\phi_{L} \left( \int_0^{s-}h(s-u)d\overline{Z}_u^{n-1}\right)\bigg|^{2}ds\bigg),\\
\end{aligned}
\end{equation*}
where we used Doob's martingale inequality and Jensen's inequality.

Note that for any real numbers $a,b>0$, 
\begin{equation}
\int_{0}^{\infty}\left|I_{z\leq a}-I_{z\leq b}\right|^{2}dz
=\int_{0}^{\infty}\left|I_{z\leq a}-I_{z\leq b}\right|dz
=|a-b|.
\end{equation}
Therefore, we get
\begin{align*}
\delta_{t}^{n}
&\leq
8\mathbb E\bigg( \int_0^t  \bigg|e^{\varphi\left(s,Z_{s-}^{i,n}\right)}\phi_{L} \left( \int_0^{s-}h(s-u)d\overline{Z}_u^{n}\right)
\\
&\qquad\qquad\qquad
-e^{\varphi\left(s,Z_{s-}^{i,n-1}\right)}\phi_{L} \left( \int_0^{s-}h(s-u)d\overline{Z}_u^{n-1}\right)\bigg| ds\bigg)\\
&
\qquad\qquad
+2T \mathbb E\bigg(\int_0^t \bigg| e^{\varphi\left(s,Z_{s-}^{i,n}\right)}\phi_{L} \left( \int_0^{s-}h(s-u)d\overline{Z}_u^{n}\right)
\\
&\qquad\qquad\qquad\qquad\qquad\qquad
- e^{\varphi\left(s,Z_{s-}^{i,n-1}\right)}\phi_{L} \left( \int_0^{s-}h(s-u)d\overline{Z}_u^{n-1}\right)\bigg|^{2}ds\bigg).
\end{align*}

By $\sup_{t\in [0,T]}\sup_{x\in \mathbb N}{|\varphi(t,x)|}<\infty$ and Lipschitz continuity of $\phi_{L}$, there exist constants $c_1,c_2\in (0,\infty)$ such that
\begin{equation*}\label{N-dim-perturbation-Hawkes-process-bounded-Thm-eq-4}
\begin{aligned}
&\bigg| e^{\varphi\left(s,Z_{s-}^{i,n}\right)}\phi_{L} \left( \int_0^{s-}h(s-u)d\overline{Z}_u^{n}\right)- e^{\varphi\left(s,Z_{s-}^{i,n-1}\right)}\phi_{L} \left( \int_0^{s-}h(s-u)d\overline{Z}_u^{n-1}\right)\bigg| 
\\
&\leq
\bigg| e^{\varphi\left(s,Z_{s-}^{i,n}\right)}-e^{\varphi\left(s,Z_{s-}^{i,n-1}\right)}\bigg|
\phi_{L} \left( \int_0^{s-}h(s-u)d\overline{Z}_u^{n}\right)
\\
&\qquad\qquad
+e^{\varphi\left(s,Z_{s-}^{i,n-1}\right)}
\bigg|\phi_L \left( \int_0^{s-}h(s-u)d\overline{Z}_u^{n}\right)
-\phi_{L} \left( \int_0^{s-}h(s-u)d\overline{Z}_u^{n-1}\right)\bigg| 
\\
\leq &c_1\left|Z^{i,n}_{s-}- Z^{i,n-1}_{s-}\right|
\left(1+\sup_{0\leq x\leq L}\phi(x)\right) +c_2\sup_{u\in[0,s)}\left|\overline{Z}_{u}^{n}- \overline{Z}_{u}^{n-1}\right|.
\end{aligned}
\end{equation*}

Thus, there exists a constant $c\in (0,\infty)$ that depends on $c_1,c_2,\sup_{0\leq x\leq L}\phi(x)$ such that
\begin{equation*}\label{N-dim-perturbation-Hawkes-process-bounded-Thm-eq-5}
\begin{aligned}
\delta^{n}_t\leq c\int_0^t \delta^{n-1}_sds,
\end{aligned}
\end{equation*}
which implies that 
\begin{equation*}
\delta^{n}_t\leq  \frac{(cNt)^n}{n!} \delta^{0}_T.  
\end{equation*}
Therefore,
\begin{equation*}
\sum_{n=1}^\infty  \delta^{n}_t<\infty.
\end{equation*}
Define 
\begin{equation*}
Z_t^{L,\varphi, N,i}:=\sum_{n=1}^\infty \left(Z^{i,n}_{t}-Z^{i,n-1}_{t}\right),~~t\in[0,T], ~i=1,\cdots,N.
\end{equation*}
Thus $\left(Z^{L,\varphi,N,1}_t,\cdots,Z^{L,\varphi,N,N}_t\right)_{t\in [0,T]}$  
is a solution of the equation  \eqref{N-dim-perturbation-truncated-Hawkes-process-eq}.
  
Next, we  prove the uniqueness of the solution of \eqref{N-dim-perturbation-truncated-Hawkes-process-eq}. Since  $\varphi$ and $\phi_L$ are bounded,   each solution  $\left(Z_t^1,\cdots, Z_t^N\right)_{t \in[0,T]}$ of   the equation  \eqref{N-dim-perturbation-truncated-Hawkes-process-eq} satisfies  
\begin{equation*}
\sum_{1\leq i \leq N}  \mathbb E\left(\left|Z_t^{i}\right|^2\right)<\infty,\qquad\text{for any $0\leq t\leq T$}.
\end{equation*}
 Let $\left(Z_t^1,\cdots, Z_t^N\right)_{t \in[0,T]}$ and $\left(Y_t^1,\cdots, Y_t^N\right)_{t \in[0,T]}$
be two solutions to  \eqref{N-dim-perturbation-truncated-Hawkes-process-eq} .  Set
\begin{equation*}
\zeta_t:=\mathbb E\left(\sup_{0\leq s\leq t} \left|Z_s^{i}-Y_s^{i}\right|^2\right).
\end{equation*}
Then, similar as in the proof for the existence, one can show that there exists some constant $c\in (0,\infty)$ such that
\begin{equation*}
\zeta_t\leq c\int_0^t \zeta_sds,\qquad t\in [0,T].
\end{equation*}
Thus,  by Gronwall's inequality, we conclude that $\zeta_T=0$. 
This implies the uniqueness of the solution of \eqref{N-dim-perturbation-truncated-Hawkes-process-eq}.

Now, let us prove that if $\left(Z^{N,1}_t,\cdots,Z^{N,N}_t\right)$  
is a solution of the equation  \eqref{N-dim-perturbation-Hawkes-process-eq}, then  for any $T>0$, we have
\begin{equation}\label{uniform:L2:finite-eq-1}
\sum_{1\leq i \leq N}  \mathbb E\left(\sup_{t\in[0,T]}\left|Z_t^{N,i}\right|^2\right)<\infty.
\end{equation}

For any $M\geq1$,  define
\begin{equation}\label{T:M:defn}
T_M:=\inf\left\{t\geq 0: \max_{1\leq i\leq N}Z_t^{N,i}\geq M\right\},
\end{equation}
and 
$$
D^{M}_t :=\mathbb E\left[\sup_{0\leq s\leq t}\left|Z_{s\wedge T_M}^{N,i}\right|^2\right]
=\mathbb{E}\left[\left|Z_{t\wedge T_M}^{N,i}\right|^{2}\right],
 $$
 which is independent of $i$. 
 
Note that  
\begin{equation*}
\sup_{t\in [0,T]}Z_t^{N,i}=Z_T^{N,i}<\infty,
\end{equation*}
and
\begin{equation*}
\lim_{M\to\infty} T_M=+\infty,\qquad\text{a.s.},
\end{equation*}
where $T_{M}$ is defined in \eqref{T:M:defn}.

Therefore, for any $t\in[0,T]$, 
\begin{equation*}\label{N-dim-perturbation-truncated-Hawkes-process-bounded-Thm-eq-1}
\begin{aligned}
D^{M}_t
&=\mathbb{E}\left[\left(\int_0^{t\wedge T_M} \int_0^\infty  I_{\left\{z \leq e^{\varphi\left(s,Z_{s-}^{N,i}\right)}\phi \left( \int_0^{s-}h(s-u)d\overline{Z}_u^{N}\right)\right\}}\pi^i(ds\,dz)\right)^{2}\right]
\\
&\leq 
2\mathbb{E}\left[\left(\int_0^{t\wedge T_M} \int_0^\infty  I_{\left\{z \leq e^{\varphi\left(s,Z_{s-}^{N,i}\right)}\phi \left( \int_0^{s-}h(s-u)d\overline{Z}_u^{N}\right)\right\}}(\pi^i(ds\,dz)-dsdz)\right)^{2}\right]
\\
&\qquad\qquad
+2\mathbb{E}\left[\left(\int_0^{t\wedge T_M} \int_0^\infty  I_{\left\{z \leq e^{\varphi\left(s,Z_{s-}^{N,i}\right)}\phi \left( \int_0^{s-}h(s-u)d\overline{Z}_u^{N}\right)\right\}}dsdz\right)^{2}\right]
\\
&=
2\mathbb{E}\left[\int_{0}^{t\wedge T_M} e^{\varphi\left(s,Z_{s-}^{N,i}\right)}\phi \left( \int_0^{s-}h(s-u)d\overline{Z}_u^{N}\right)ds\right]
\\
&\qquad\qquad
+2\mathbb{E}\left[\left(\int_{0}^{t\wedge T_M} e^{\varphi\left(s,Z_{s-}^{N,i}\right)}\phi \left( \int_0^{s-}h(s-u)d\overline{Z}_u^{N}\right)ds\right)^{2}\right]
\\
&\leq
 2\int_0^t \mathbb E\bigg(  e^{\varphi\left({s\wedge T_M} ,Z_{{s\wedge T_M} -}^{N,i}\right)}\phi \left( \int_0^{{s\wedge T_M}-}h({s\wedge T_M}-u)d\overline{Z}_u^{N}\right)  \bigg)ds\\
&\qquad\qquad
+2T\int_0^t \mathbb E\bigg( \bigg| e^{\varphi\left({s\wedge T_M},Z_{{s\wedge T_M}-}^{N,i}\right)}\phi\left( \int_0^{{s\wedge T_M}-}h({s\wedge T_M}-u)d\overline{Z}_u^{N}\right) \bigg|^2  \bigg)ds\\
\leq& c \int_0^t \left(D^{M}_s+1\right)ds,
\end{aligned}
\end{equation*}
for some constant $c>0$ independent of $M$ and $N$,
where we used the Jensen's inequality.
Thus,
\begin{equation*}
D:=\sup_{N\geq 1, M\geq 1}D^{M}_T\leq  e^{cT},
\end{equation*}
and 
\begin{equation}\label{uniform:L2:finite-eq-2} 
\sup_{N\geq 1}\max_{1\leq i\leq N}\mathbb E\left[\sup_{0\leq t\leq T}\left|Z_s^{N,i}\right|^2\right]\leq D.
\end{equation}
which implies \eqref{uniform:L2:finite-eq-1}.
 
Next, we use the the equation  \eqref{N-dim-perturbation-truncated-Hawkes-process-eq} to approximate  \eqref{N-dim-perturbation-Hawkes-process-eq}.  Set
$$
\tau_{L}:=\inf\left\{t>0:\overline{Z}_{t}^{L,\varphi,N}\geq L\right\}.
$$
Then for any $L_1\geq L$, on $\{\tau_{L}> T\}$,    ${Z}_{t}^{L,\varphi,N}={Z}_{t}^{L_1,\varphi,N}$ for $t\in [0,T]$. Thus 
$$
\tau_{L_1}\geq \tau_L~~a.s.
$$ 
%Since $\sup_{0\leq t\leq T}\int_{0}^{t}h(t-s)d\overline{Z}_{s}^{L,\varphi,N}
%\leq\left(h(0)+\int_{0}^{T}|h'(t)|dt\right)\overline{Z}_{T}^{L,\varphi,N}$,. 
Noting that for any $T>0$, 
$$
\mathbb P\left(\tau_L\leq T\right)\leq \mathbb P\left(\sup_{0\leq t\leq T}\overline{Z}_{t\wedge \tau_L}^{L,\varphi,N}\geq L\right)\leq \frac{\sqrt{D}}{L},
$$
we have that  $\sum_{k=1}^\infty\mathbb P\left(\tau_{2^k}\leq T\right)<\infty$.   Since   $\{\tau_{2^k},k\geq 1\}$  is  a.s. non-decreasing,   as $k\to\infty$,   
$$
\tau_{2^k}\to +\infty~~   a.s.  
$$
 Thus, the equation \eqref{N-dim-perturbation-Hawkes-process-eq}  has a strong solution ${Z}_{t}^{\varphi,N}$ satisfying that for any $k\geq 1$,
 $$
 {Z}_{t}^{\varphi,N}={Z}_{t}^{2^k,\varphi,N} \mbox{ for all } t\in [0,T],  \mbox{ on } \{\tau_{2^k}> T\}.
 $$
By \eqref{uniform:L2:finite-eq-2}, the solution satisfies  \eqref{uniform:L2:finite}.

Finally, let us prove uniqueness of the solution of \eqref{N-dim-perturbation-Hawkes-process-eq}.  Let $X_{t}^{\varphi,N}$ be another solution of  \eqref{N-dim-perturbation-Hawkes-process-eq}.  Define
$$
\sigma_L=\inf\left\{t\geq 0:\overline{X}_{t}^{\varphi,N}\geq L\right\}.
$$
Then, for any $T>0$, on $\{\sigma_{2^k}> T\}$,  $X_{t}^{\varphi,N}=Z_{t}^{L,\varphi,N}$ for any $t\in [0,T]$. By \eqref{uniform:L2:finite-eq-2},  we also have that
$$
\sigma_{2^k}\to +\infty~~   a.s.  
$$
Therefore,  $X_{t}^{\varphi,N}=Z_{t}^{L,\varphi,N}$ for any $t\in [0,T]$. This completes the proof.
\end{proof}

Similarly, we have the following result:

\begin{lem}\label{N-dim-perturbation-Hawkes-process-bounded-Thm-2}  Consider the perturbed SDE:
\begin{equation} \label{Hawkes-perturbation-mean-eq-1}
\tilde Z_t= \int_0^t\int_0^\infty I_{ \left\{z \leq e^{\varphi\left(s,\tilde Z_{s-}\right)} \phi\left(\int_0^s 
h(s-u)d \mathbb E(\tilde Z_u)\right)\right\}}\pi(ds\,dz),\;\;\text{ }\;\;t\geq 0,
\end{equation}
Suppose (A.1), (A.2)  and   (A.3).  If  $\varphi$ satisfies \eqref{N-dim-perturbation-Hawkes-process-varphi-bounded-c},
then, the equation \eqref{Hawkes-perturbation-mean-eq-1}  has a unique solution   $\left(\tilde Z^{\varphi}_t\right)_{t\in [0,T]}$ in $D\left([0,T],\mathbb N\right)$.   
Moreover, the solution satisfies  
$$
\mathbb E\left(\left|\tilde Z^{\varphi}_t\right|^2\right)<\infty,~t\geq 0.
$$
\end{lem}

\begin{thm}\label{N-dim-perturbation-Hawkes-process-exp-Thm-1}
Let  (A.1), (A.2)  and   (A.3) hold.  Assume 
\begin{equation} \label{N-dim-perturbation-Hawkes-process-varphi-exp-c}
\int_0^T e^{\varphi^*(t)}
dt<\infty,
\end{equation}
where $\varphi^*(t):=\sup_{x\in\mathbb N} \varphi(t,x)$.
Set  $\varphi_n(t,x):=((-n)\vee \varphi(t,x))\wedge n, ~n\geq 1$. 
 Define 
$$
Z^{(n)}_t = \left(Z^{(n),1}_t,\cdots,Z^{(n),N}_t\right):= \left(Z^{\varphi_n,N,1}_t,\cdots,Z^{\varphi_n,N,N}_t\right), 
 $$
 that is,  $Z^{(n)}_t = \left(Z^{(n),1}_t,\cdots,Z^{(n),N}_t\right)$ is the unique solution of the stochastic differential equation:
\begin{equation} \label{N-dim-perturbation-Hawkes-process-eq-n}
 Z^{(n),i}_t=\int_0^t \int_0^\infty  I_{\left\{z \leq e^{\varphi_{n}\left(s,Z_{s-}^{(n),i}\right)}\phi \left( \int_0^{s-}h(s-u)d\overline{Z}_u^{(n)}\right)\right\}}\pi^i(ds\,dz),~~i=1,\cdots,N,
\end{equation}
where $\overline{Z}_t^{(n)}=\frac{1}{N}\sum_{i=1}^N  Z^{(n),i}_t$.  
Then $\left\{Z^{(n)}_t, t\in [0,T]\right\},n\geq 1$ is tight  in $D\left([0,T], \mathbb N^N\right)$, 
and  its each limit point $Z_t^{\varphi,N}$ is  a weak solution  
of the equation \eqref{N-dim-perturbation-Hawkes-process-eq}.   
Moreover, each weak solution $Z_t^{\varphi,N}$ of 
the equation \eqref{N-dim-perturbation-Hawkes-process-eq}  satisfies  
\begin{equation*}
\sup_{1\leq i\leq N}\mathbb E\left[Z_t^{\varphi,N,i}\right]<\infty,~t\geq 0.
\end{equation*}
\end{thm}

\begin{proof} There exists a constant $c\in(0,\infty) $ such that
\begin{equation*}\label{N-dim-perturbation-Hawkes-process-exp-Thm-eq-1}
\begin{aligned}
D^{(n)}_t:
&=\mathbb E\left(Z^{(n),i}_t\right)=\mathbb{E}\left(\int_0^t \int_0^\infty  I_{\left\{z \leq e^{\varphi\left(s,Z_{s-}^{(n),i}\right)}\phi \left( \int_0^{s-}h(s-u)d\overline{Z}_u^{(n)}\right)\right\}}\pi^i(ds\,dz)\right)
\\
&=
 \int_0^t \mathbb E\bigg(  e^{\varphi\left(s,Z_{s-}^{(n),i}\right)}\phi \left( \int_0^{s-}h(s-u)d\overline{Z}_u^{(n)}\right)  \bigg)ds\\
&\leq c \int_0^Te^{\varphi_n^*(s)}ds+c \int_0^te^{\varphi_n^*(s)}D^{(n)}_sds.
\end{aligned}
\end{equation*}
Since  $ \int_0^Te^{\varphi_n^*(s)}ds\to  \int_0^Te^{\varphi^*(s)}ds<\infty $,  we have
$$
D:=\sup_{n\geq 1}D^{n}_T\leq  \sup_{n\geq 1}\int_0^Te^{\varphi_n^*(s)}ds e^{c\int_0^Te^{\varphi_n^*(s)}ds}<\infty,
$$
and it follows that
\begin{equation}\label{N-dim-perturbation-Hawkes-process-exp-Thm-eq-3}
\begin{aligned}
\lim_{M\to\infty }\mathbb P\left(\left|Z^{(n)}_T\right|\geq M\right) \leq \lim_{M\to\infty } \frac{ND}{M}=0.
\end{aligned}
\end{equation}

We also have that for any $0\leq s<t\leq T$,
\begin{equation*}\label{N-dim-perturbation-Hawkes-process-exp-Thm-eq-2}
\begin{aligned}
\left|Z^{(n),i}_t-Z^{(n),i}_s\right|
&= \int_s^t \int_0^\infty  I_{\left\{z \leq e^{\varphi_n\left(v,Z_{s-}^{(n),i}\right)}\phi \left( \int_0^{v-}h(v-u)d\overline{Z}_u^{(n)}\right)\right\}}\pi^i(dv\,dz)\\
&\leq c\left(1+  \overline{Z}_T^{(n)}\right)\int_s^t   e^{\varphi_n^*(v)}  dv+\left|M_t^{(n),i}-M_s^{(n),i}\right|,
\end{aligned}
\end{equation*}
where
$$
M_t^{(n),i}:=\int_0^t \int_0^\infty  I_{\left\{z \leq e^{\varphi_n\left(v,Z_{v-}^{(n),i}\right)}\phi \left( \int_0^{v-}h(v-u)d\overline{Z}_u^{(n)}\right)\right\}}(\pi^i(dv\,dz)-dvdz).
$$
Then, there exists a constant $c\in(0,\infty) $ such that
$$
\begin{aligned}
\left\<M^{(n),i}\right\>_t-\left\<M^{(n),i}\right\>_s=& \int_s^{t}e^{\varphi_n\left(v,Z_{v-}^{(n),i}\right)}\phi\left(\int_0^{s-} 
h(v-u)d \overline{Z}^{(n)}_u\right)  dv\\
\leq &  c\left(1+  \overline{Z}_T^{(n)}\right)\int_s^{t}e^{\varphi_n^*(v)}  dv.
\end{aligned}
$$
Thus, for any $\epsilon>0$
$$
\begin{aligned}
&\lim_{\delta\to 0}\limsup_{n\to\infty}\mathbb P\left(\sup_{|s-t|\leq \delta, 0\leq s,t\leq T}\left|\left\<M^{(n),i}\right\>_t-\left\<M^{(n),i}\right\>_s\right|>\epsilon \right)\\
\leq &\lim_{\delta\to 0}\limsup_{n\to\infty}\mathbb P\left(\sup_{|s-t|\leq \delta, 0\leq s,t\leq T}c\left(1+  \overline{Z}_T^{(n)}\right)\int_s^{t}e^{\varphi_n^*(v)}  dv>\epsilon \right)\\
\leq &\frac{c(1+D)}{\epsilon}\lim_{\delta\to 0}\limsup_{n\to\infty}  \sup_{|s-t|\leq \delta, 0\leq s,t\leq T}\int_s^{t}e^{\varphi_n^*(v)}  dv=0.
\end{aligned}
$$
By a criterion for tightness of locally square-integrable martingales (cf.  Theorem VI. 4.13 in \cite{JacodShiryaev}),    $\left\{M^{(n)}_t=\left(M_t^{(n),1}, \cdots, M_t^{(n),N}\right), t\in [0,T]\right\},n\geq 1$ is tight  in $D\left([0,T], \mathbb N^N\right)$.  Therefore, $\left\{Z^{(n)}_t, t\in [0,T]\right\},n\geq 1$ is tight  in $D\left([0,T], \mathbb N^N\right)$. Let   $Z_t $ be  any limit point. Then, for any $i=1,\cdots, N$,   
$$
M_t^i:=Z_t^i-\int_0^t  e^{\varphi\left(v,Z_{v-}^{i}\right)}\phi \left( \int_0^{v-}h(v-u)d\overline{Z}_u\right)dv
$$
is martingale, and
$$
\left\<M^i,M^j\right\>_t
=\left\{ \begin{array}{ll} 0& \mbox{ if } i\not=j,\\
                                   \int_0^t  e^{\varphi\left(v,Z_{v-}^{i}\right)}\phi \left( \int_0^{v-}h(v-u)d\overline{Z}_u\right)dv &\mbox{ if } i=j.
                                   \end{array}
\right.
$$\
By the representation theorem, there exist $N$-independent Poisson point processes $\tilde{\pi}^i$ such that
\begin{align*}
&Z_t^i-\int_0^t  e^{\varphi\left(v,Z_{v-}^{i}\right)}\phi \left( \int_0^{v-}h(v-u)d\overline{Z}_u\right)dv
\\
&=\int_0^t \int_0^\infty  I_{\left\{z \leq e^{\varphi\left(v,Z_{v-}^{i}\right)}\phi \left( \int_0^{v-}h(v-u)d\overline{Z}_u\right)\right\}}(\tilde{\pi}^i(dv\,dz)-dvdz),
\end{align*}
that is,
\begin{equation*}
Z_t^i=\int_0^t \int_0^\infty  I_{\left\{z \leq e^{\varphi\left(v,Z_{v-}^{i}\right)}\phi \left( \int_0^{v-}h(v-u)d\overline{Z}_u\right)\right\}}\tilde{\pi}^i(dv\,dz),~~i=1,\ldots,N.
\end{equation*}
This completes the proof.
\end{proof}
 
\begin{thm}\label{N-dim-perturbation-Hawkes-process-exp-Thm-2}  Suppose (A.1), (A.2)  and   (A.3) hold.  Assume that $\varphi$ satisfies  \eqref{N-dim-perturbation-Hawkes-process-varphi-exp-c}.

(1).  Set  $\varphi_n(t,x)=((-n)\vee \varphi(t,x))\wedge n, ~n\geq 1$. 
Let
$$
\tilde Z^{(n)}_t := \tilde Z^{\varphi_n}_t, 
 $$
be the unique solution of the stochastic differential equation:
\begin{equation}\label{Hawkes-perturbation-mean-eq-n}
\tilde  Z^{(n)}_t=\int_0^t \int_0^\infty  I_{\left\{z \leq e^{\varphi_n\left(s,\tilde Z_{s-}^{(n),i}\right)}\phi \left( \int_0^{s-}h(s-u)d\mathbb{E}({\tilde Z}_u^{(n)})\right)\right\}}\pi(ds\,dz).
\end{equation}
Then $\left\{\tilde Z^{(n)}_t, t\in [0,T]\right\},n\geq 1$ is tight  in $D\left([0,T], \mathbb N^N\right)$, and  its each limit point $\tilde Z_t$ is  a weak solution  
of the equation \eqref{Hawkes-perturbation-mean-eq-1}.   Moreover, any weak solution  $\tilde Z_t$ of the equation \eqref{Hawkes-perturbation-mean-eq-1} satisfies  
$$
\mathbb E\left(\tilde{Z}_t\right)<\infty,~t\geq 0,
$$
and the distribution $\mu_t^{\varphi}(dx)$ of any weak solution  $\tilde Z_t$ satisfies
\begin{equation} \label{perturbation-mean-fields-equation}
\begin{aligned}
\<\mu_t,\psi(t)\> = &\psi(0,0)+\int_0^t   \left\<\mu_s,\partial_s \psi(s)\right\> ds\\
&+\int_0^t   \left\<\mu_s,\nabla \psi(s)e^{\varphi(s)}\right\> \phi\left( \int_0^s h(s-u)d\bar{\mu}_u\right)ds,~\psi\in C^{1,lip}([0,T]\times\mathbb N).
\end{aligned}
\end{equation}

(2).  The equation \eqref{perturbation-mean-fields-equation} has a unique solution  in $\mu\in M[0,T]$. We denote the solution by $\mu_t^{\varphi}(dx)$.  

In particular,  the stochastic differential equation \eqref{Hawkes-perturbation-mean-eq-1} has a unique weak solution $\left\{\tilde Z_t^\varphi, t\in[0,T]\right\}$,  and $\left\{\tilde Z^{(n)}_t, t\in [0,T]\right\},n\geq 1$ converges weakly to  $\left\{\tilde Z_t^\varphi, t\in[0,T]\right\}$  in $D\left([0,T], \mathbb N^N\right)$.  
\end{thm}

\begin{proof} The proof of (1) is similar to that of Theorem \ref{N-dim-perturbation-Hawkes-process-exp-Thm-1}. 
Thus, we only focus on the proof of (2). 
By It\^{o}'s formula, the distribution $\mu_t^{\varphi}(dx)$ of any weak solution  $Z_t$   satisfies the equation \eqref{perturbation-mean-fields-equation}.
Next, let us  prove uniqueness of the equation \eqref{perturbation-mean-fields-equation}. Let $\mu$ and $\nu$
be two solutions to \eqref{perturbation-mean-fields-equation}.  Set
\begin{equation*}
\begin{aligned}
\eta_t:=& \sup_{0\leq s\leq t}\left(|\mu_s(\{0\}) -\nu_s(\{0\})|+\sum_{n=1}^\infty n|\mu_s(\{n\}) -\nu_s(\{n\})|\right)\\
\leq &\bar{\mu}_t+\bar{\nu}_t+\sup_{0\leq s\leq t}\left(\mu_s(\{0\}) +\nu_s(\{0\}\right).
\end{aligned}
\end{equation*}
By taking $\psi(t,x)=I_{\{n\}}(x)$ in \eqref{perturbation-mean-fields-equation}, 
it follows that there exists a constant $c\in (0,\infty)$ such that for any $n\geq 1$,
$$
\begin{aligned}
|\mu_t(\{n\}) -\nu_t(\{n\})|\leq & c\int_0^t e^{\varphi^*(s)}\left(|\mu_s(\{n-1\}) -\nu_s(\{n-1\})|+|\mu_s(\{n\}) -\nu_s(\{n\})|\right)ds\\
&+c\int_0^t e^{\varphi^*(s)}(\mu_s(\{n-1\})+\mu_s(\{n\})+\nu_s(\{n-1\})+\nu_s(\{n\}))\eta_s ds,
\end{aligned}
$$
and
$$
\begin{aligned}
|\mu_t(\{0\}) -\nu_t(\{0\})|\leq & c\int_0^t e^{\varphi^*(s)}|\mu_s(\{0\}) -\nu_s(\{0\})|ds\\
&\qquad+c\int_0^t e^{\varphi^*(s)}(\mu_s(\{0\})+\nu_s(\{0\}))\eta_s ds,
\end{aligned}
$$
which implies that  there exists a constant $c\in (0,\infty)$ such that
\begin{equation*}
\eta_t\leq  c(\bar{\mu}_T+\bar{\nu}_T+1) \int_0^te^{\varphi^*(s)} \eta_s ds.
\end{equation*}
Thus, by Gronwall's inequality, we conclude that $\eta_T=0$. This completes the proof.
\end{proof}
 
Next, we obtain a law of large numbers result for the perturbed mean-field model.

\begin{thm} \label{perturbation-mean-fields-LLN-thm-1}Suppose (A.1), (A.2)  and   (A.3) hold. Assume that $\varphi$ satisfies \eqref{N-dim-perturbation-Hawkes-process-varphi-exp-c}.  Let  $\left\{Z^{\varphi, N,i}_t,t\in[0,T]\right\}$ be a weak solution of the SDEs \eqref{N-dim-perturbation-Hawkes-process-eq} and let $\mu_t^{\varphi}(dx)$ be the unique solution of the equation \eqref{Hawkes-perturbation-mean-eq-1}.  Set
\begin{equation} \label{N-dim-perturbation-mean-fields-eq}
L^{\varphi, N}_t:=\frac{1}{N}\sum_{i=1}^N\delta_{Z^{\varphi, N,i}_t}.
\end{equation}
Then for any   an open subset $O\ni (\mu_t^{\varphi})_{t\in[0,T]} $,
\begin{equation} \label{perturbation-mean-fields-LLN-thm-1-eq-1}
\lim_{N\to\infty}\mathbb P\left(L^{\varphi, N}_t\in O\right)=1.
\end{equation}
\end{thm}

\begin{proof}
For any $\psi \in C^{lip}(\mathbb N  )$, by It\^{o}'s formula,
\begin{align*}
\left\<L^{\varphi, N}_t,\psi\right\>
&= \frac{1}{N}\sum_{i=1}^N\psi\left(Z^{\varphi, N}_t\right)
\\
&=\psi(0)+
\int_0^{t}\left\<L_s^{\varphi,N}, \nabla \psi\right\>  e^{\varphi(s)} \phi\left(\int_0^{s-} 
h(s-u)d \overline{Z}^{\varphi, N}_u\right)  ds+M_t^{\varphi,N},
\end{align*}
where we recall from \eqref{N-dim-perturbation-Hawkes-process-mean-eq} that
\begin{equation} \label{N-dim-perturbation-mean-process-eq}
\overline{Z}^{\varphi, N}_t=\frac{1}{N}\sum_{i=1}^NZ^{\varphi, N,i}_t,
\end{equation}
and $M_t^{\varphi,N}$ is a martingale, given by
\begin{equation*}
M_t^{\varphi,N}:=\frac{1}{N}\sum_{i=1}^N\int_0^{t-} \int_0^\infty\nabla\psi\left(Z^{\varphi, N,i}_{s-}\right)  
I_{ \left\{z \leq e^{\varphi\left(s,Z^{\varphi, N,i}_{s-}\right)} \phi\left(\int_0^{s-} 
h(s-u)d \overline{Z}^{\varphi,N}_u\right)\right\}} ({\pi}^{i}(ds\,dz)-dsdz).
\end{equation*}
Then, there exists a constant $c\in(0,\infty) $ such that
$$
\begin{aligned}
\left\<M^{\varphi,N}\right\>_t\leq &\frac{c}{N} \int_0^{t}e^{\varphi^*(s)}\phi\left(\int_0^{s-} 
h(s-u)d \overline{Z}^{\varphi,N}_u\right)  ds.
\end{aligned}
$$
Thus by Doob's martingale inequality, for any $\epsilon>0$,
\begin{equation}
\mathbb{P}\left(\sup_{0\leq t\leq T}\left|M_{t}^{\varphi,N}\right|\geq\epsilon\right)
\leq
\frac{4\mathbb{E}\left[\left|M_{T}^{\varphi,N}\right|^{2}\right]}{\epsilon^{2}}
=\frac{4\mathbb{E}\left[\left\langle M^{\varphi,N}\right\rangle_{T}\right]}{\epsilon^{2}}.
\end{equation}
Moreover, we can compute that
\begin{align*}
\mathbb{E}\left[\left\langle M^{\varphi,N}\right\rangle_{T}\right]
&\leq
\frac{c}{N} \int_0^{t}e^{\varphi^*(s)}
\left(\phi(0)+\alpha\mathbb{E}\left|\int_0^{s-} 
h(s-u)d\overline{Z}^{\varphi,N}_u\right|\right)ds
\\
&\leq
\frac{c}{N} \int_0^{t}e^{\varphi^*(s)}ds
\cdot\left(\phi(0)+\alpha\left(h(0)+\int_{0}^{T}|h'(u)|du\right)\mathbb{E}\left[\overline{Z}^{\varphi,N}_T\right]\right),
\end{align*}
and we recall that $\overline{Z}_{t}^{\varphi,N}=\frac{1}{N}\sum_{i=1}^{N}Z_{t}^{\varphi,N,i}$
where
\begin{align*}
\mathbb{E}\left[\overline{Z}_{t}^{\varphi,N}\right]
&=\mathbb{E}\left[Z_{t}^{\varphi,N,i}\right]
\\
&=\int_{0}^{t}\mathbb{E}\left[e^{\varphi\left(s,Z_{s-}^{\varphi,N,i}\right)}
\phi\left(\int_{0}^{s-}h(s-u)d\overline{Z}_{u}^{\varphi,N}\right)\right]ds
\\
&\leq\int_{0}^{t}e^{\varphi^{\ast}(s)}
\left(\phi(0)+
\alpha\mathbb{E}\left|\int_{0}^{s-}h(s-u)d\overline{Z}_{u}^{\varphi,N}\right|\right)ds
\\
&
\leq\int_{0}^{t}e^{\varphi^{\ast}(s)}
\left(\phi(0)+
\alpha\left(h(0)+\int_{0}^{s}|h'(u)|du\right)\mathbb{E}\left[\overline{Z}_{s}^{\varphi,N}\right]\right)ds,
\end{align*}
and by Gronwall's inequality, we get
\begin{equation*}
\mathbb{E}\left[\overline{Z}_{t}^{\varphi,N}\right]
\leq\phi(0)\int_{0}^{t}e^{\varphi^{\ast}(s)}ds
\cdot 
e^{\alpha((h(0)+\int_{0}^{T}|h'(u)|du)t}.
\end{equation*}
Hence, we have proved that
\begin{equation*}
\sup_{0\leq t\leq T}\left|M_{t}^{\varphi,N}\right|\rightarrow 0,
\qquad\mbox{in probability  as } N\to\infty.
\end{equation*}
For any $\epsilon>0$
$$
\begin{aligned}
&\lim_{\delta\to 0}\limsup_{N\to\infty}\mathbb P\left(\sup_{|s-t|\leq \delta, 0\leq s,t\leq T}\left|\left\<L^{\varphi, N}_t,\psi\right\>-\left\<L^{\varphi, N}_s,\psi\right\>\right|>\epsilon \right)\\
\leq &\lim_{\delta\to 0}\limsup_{N\to\infty}\mathbb P\left(\sup_{|s-t|\leq \delta, 0\leq s,t\leq T}\|\psi\|_{lip} \int_s^{t} e^{\varphi^*(v)}\phi\left(\int_0^{v-} 
h(v-u)d \overline{Z}^{\varphi, N}_u\right)  dv>\epsilon/2 \right)=0.
\end{aligned}
$$
Therefore, $\left\{L^{\varphi, N}\right\}$ is tight. Let $\mu$ be any limit point. Then $\mu$ satisfies \eqref{perturbation-mean-fields-equation}.  By Theorem~\ref{N-dim-perturbation-Hawkes-process-exp-Thm-2} ,  $\mu=\mu^{\varphi}$. The proof is complete.
\end{proof}

%%%%%%%%%%%%%%%%%%%%%%%%%%%%%%%%%%%%%%%%
\section{The rate function of large deviations}\label{sec:rate:function}

In this section, we discuss and analyze the rate function of the large deviations.
Given  $\mu\in M[0,T]$,  for any  finite measurable function $\varphi $ with 
$$
\int_0^T\int_0^\infty e^{\varphi(t,x)}\phi\left(\int_0^{t} h(t-s)d\bar{\mu}_s\right)\mu_t(dx)dt<\infty, 
$$
define
\begin{equation}\label{l-normal-def-eq-0}
\begin{aligned}
\ell_\mu(\varphi) :=& \int_0^T\sum_{n=0}^\infty \partial_tf_\mu(t,n)\varphi(t,n)dt
=-\int_0^T\sum_{n=0}^\infty\sum_{k=n}^{\infty}\varphi(t,k)\partial_tF_\mu(t,n)dt.
\end{aligned}
\end{equation}
and 
\begin{equation}\label{J-funtion-def-eq-1}
\begin{aligned}
J_\mu(\varphi):=&\ell_\mu(\varphi)-\int_0^T\int_0^\infty \left(e^{\varphi(t,x)}-1 \right)\phi\left(\int_0^{t} h(t-s)d\bar{\mu}_s\right)\mu_t(dx)dt.
\end{aligned}
\end{equation}

Then when $\varphi\in C^{1,b}([0,T]\times\mathbb N )$,  
\begin{equation}\label{l-normal-def-eq-1}
\begin{aligned}
\ell_\mu(\varphi)=&\langle \mu(T),\varphi(T)\rangle-\varphi(0,0) -\int_0^T\left\<\mu(t), \partial _t\varphi(t) \right\> dt.
\end{aligned}
\end{equation}

Define
\begin{equation}\label{ldp-rate-funtion-def-eq-1}
\begin{aligned}
I(\mu):=\sup\left\{J_\mu(\varphi);~ \varphi\in C^{1,b}([0,T]\times\mathbb N )   \right\}.
\end{aligned}
\end{equation}

\begin{lem}\label{ldp-rate-funtion-Property-1-lem-1}
Suppose (A.1), (A.2)  and   (A.3) hold. Let  $\mu\in M[0,T]$ satisfy $ 
I(\mu) <\infty,
$

(1).  For each $x\in\mathbb N$, $t\mapsto 1-F_\mu(t,x)$ is absolutely continuous w.r.t. $f_\mu(t,x)dt $.  
In particular,  for each $x\in\mathbb N$,  $t\mapsto  \mu(\{x\}) $ is continuous.

(2). There exists a  function $\varphi=\varphi_\mu\in B([0,T]\times\mathbb N )$   such that
  for any $\psi\in C_K([0,T]\times\mathbb N  )$, i.e. the space with compact support in $x$,
\begin{equation} \label{general-perturbation-mean-fields-equation}
\begin{aligned}
\<\mu_t,\psi(t)\> = &\psi(0,0)+\int_0^t   \left\<\mu_s,\partial_s \psi(s)\right\> ds\\
&+\int_0^t   \left\<\mu_s,\nabla \psi(s)e^{\varphi(s)}\right\> \phi\left( \int_0^s h(s-u)d\bar{\mu}_u\right)ds.
\end{aligned}
\end{equation}
Moreover, $\varphi=\varphi_\mu$ is determined uniquely by
\begin{equation} \label{general-perturbation-mean-fields-equation-sol}
{f_\mu(t,x)\phi\left( \int_0^th(t-u)d\bar{\mu}_u\right)}e^{\varphi_\mu(t,x)}= -\partial_tF_\mu(t,x),~x\in\mathbb N, t\in[0,T], 
\end{equation}
where we set $\varphi_\mu(t,x)=0$ if $f_\mu(t,x)=0$. 
The following formula also holds
\begin{equation}\label{ldp-rate-funtion-representation-eq-0}
\int_0^t \int_0^\infty e^{\varphi_\mu(s,x)}
\phi\left( \int_0^s h(s-u)d\bar{\mu}_u\right)\mu_s(dx)ds= \int_0^\infty x\mu_t(dx).
\end{equation}
\end{lem}

\begin{proof}(1). Assume that there exists $n\in\mathbb N$, so that 
$-\partial_tF_\mu(t,n)dt$ is not absolutely continuous w.r.t. $f_\mu(t,n)dt $.  Choose measurable subset $A\subset [0,T]$ such that  $\int_A f_\mu(t,n)dt=0$ and $\delta:=- \int_A \partial_t  F_\mu(t,n)dt>0$.  
Take $\varphi_\lambda(t,x)=\lambda I_A(t) I_{\{n\}}(x)$, $\lambda>0$. Then as $\lambda\to \infty$,
$$
I(\mu)\geq J_\mu(\varphi_\lambda)=\lambda \delta \to \infty.
$$
Thus, $I(\mu)=\infty$.

(2).  
 By \eqref{general-perturbation-mean-fields-equation}, we have
 \begin{equation*} \label{general-perturbation-mean-fields-equation-eq-2}
 \left\{
 \begin{aligned}
f_\mu(t,0)=&1-\int_0^tf_\mu(s,0) e^{\varphi(s,0)}  \phi\left( \int_0^sh(s-u)d\bar{\mu}_u\right)ds, \\
f_\mu(t,n)= &\int_0^t   \left(f_\mu(s,n-1)e^{\varphi(s,n-1)}-f_\mu(s,n)e^{\varphi(s,n)}\right)
\phi\left( \int_0^sh(s-u)d\bar{\mu}_u\right)ds,~n\geq 1.
\end{aligned}
\right.
\end{equation*}

Then  
 \begin{equation*} \label{ldp-rate-funtion-Property-1-eq-3}
 \left\{
 \begin{aligned}
\partial_tf_\mu(t,0)=& -f_\mu(t,0)e^{\varphi(t,0)}  \phi\left( \int_0^th(t-u)d\bar{\mu}_u\right), \\
\partial_tf_\mu(t,n) = & \left(f_\mu(t,n-1)e^{\varphi(t,n-1)}-f_\mu(t,n)e^{\varphi(t,n)}\right)\phi\left( \int_0^th(t-u)d\bar{\mu}_u\right),~n\geq 1.
\end{aligned}
\right.
\end{equation*}
 Therefore, by summing over $n$,
\eqref{general-perturbation-mean-fields-equation-sol} holds.
We recall that $F_{\mu}(t,x)=\mu_{t}([0,x])$ and $f_{\mu}(t,x)=\mu_{t}(\{x\})$.
By summing over $x$ in \eqref{general-perturbation-mean-fields-equation-sol}, we get
\begin{equation*}
\int_{0}^{\infty}\phi\left( \int_0^th(t-u)d\bar{\mu}_u\right)e^{\varphi_\mu(t,x)}\mu_{t}(dx)
=\partial_{t}\int_{0}^{\infty}x\mu_{t}(dx).
\end{equation*}
By integrating from $0$ to $t$, we get \eqref{ldp-rate-funtion-representation-eq-0}.

Finally, 
\eqref{general-perturbation-mean-fields-equation} can be obtained by  \eqref{general-perturbation-mean-fields-equation-sol}.
This completes the proof.
\end{proof}

Define
\begin{equation*}
G_\mu(t,x):=\frac{-\partial_tF_\mu(t,x)}{f_\mu(t,x)\phi\left( \int_0^th(t-u)d\bar{\mu}_u\right)},
\end{equation*}
where  we set $0/0:=1$. Then
\begin{equation*}\label{jensen}
\int_0^T\int_0^\infty G_\mu(t,x)\phi\left( \int_0^th(t-u)d\bar{\mu}_u\right)\mu_t(dx)dt=\int_0^\infty x\mu_T(dx).
\end{equation*}

\begin{thm}\label{ldp-rate-funtion-Property-1}
Suppose (A.1), (A.2)  and   (A.3) hold.  
Let  $\mu\in M[0,T]$. Then $I(\mu)<\infty$ if and only if  
\begin{equation}\label{ldp-rate-funtion-representation-eq-2}
\int_0^T\int_{0}^\infty \left( {G_\mu(t,x)} \log  {G_\mu(t,x)} \right)\phi\left( \int_0^th(t-u)d\bar{\mu}_u\right)\mu_t(dx)dt<\infty.
\end{equation}
Furthermore,  in this case,
\begin{equation}\label{ldp-rate-funtion-representation-eq-3}
\begin{aligned}
I(\mu) =&\int_0^T\int_{0}^\infty \bigg( {G_\mu(t,x)} \log {G_\mu(t,x)} -{G_\mu(t,x)}+1\bigg)  \phi\left(\int_0^{t} h(t-s)d\bar{\mu}_s\right)\mu_t(dx)dt.
\end{aligned}
\end{equation}
That is,
\begin{equation}\label{ldp-rate-funtion-representation-eq-1}
\begin{aligned}
I(\mu) =& \int_0^T\int_{0}^\infty \left(\varphi_\mu(t,x)e^{\varphi_\mu(t,x)}-e^{\varphi_\mu(t,x)}+1\right)\phi\left(\int_0^{t} h(t-s)d\bar{\mu}_s\right)\mu_t(dx)dt,
\end{aligned}
\end{equation}
where  $\varphi=\varphi_\mu\in  B([0,T]\times\mathbb N  )$  is the solution of \eqref{general-perturbation-mean-fields-equation}.  
 \end{thm}
 
\begin{proof}  For any given $t$ and $x$, let us define the function:
\begin{equation*}
F(y):=y(-\partial_{t}F_{\mu}(t,x))
-(e^{y}-1)\phi\left(\int_{0}^{t}h(t-u)d\bar{\mu}_u\right)\mu_{t}(\{x\}),
\end{equation*}
where $-\infty<y<\infty$. Note that  when $\partial_t F_t(t,x)\not=0$ and $\mu_t(\{x\})\not=0$, the function $F(y)$ is maximized at
\begin{equation*}
y=\log\left(\frac{(-\partial_{t}F_{\mu}(t,x))}{\phi\left(\int_{0}^{t}h(t-u)d\bar{\mu}_u\right)\mu_{t}(\{x\})}\right)
=\log\left(G_{\mu}(t,x)\right).
\end{equation*}
Therefore, for any $\varphi(t,x)\in C^{1,lip}([0,T]\times\mathbb N)$,  by definition, we have
\begin{align*}
J_{\mu}(\varphi)
&=\int_{0}^{T}\sum_{x=0}^{\infty}\varphi(t,x)(-\partial_{t}F_{\mu}(t,x))dt
\\
&\qquad\qquad\qquad
-\int_{0}^{T}\sum_{x=0}^{\infty}(e^{\varphi(t,x)}-1)
\phi\left(\int_{0}^{t}h(t-u)d\bar{\mu}_u\right)\mu_{t}(\{x\})dt\\
&\leq
\int_0^T\int_{0}^\infty \bigg( {G_\mu(t,x)} \log {G_\mu(t,x)} -{G_\mu(t,x)}+1\bigg)  \phi\left(\int_0^{t} h(t-s)d\bar{\mu}_s\right)\mu_t(dx)dt.
\end{align*}
Hence, if
$$
\int_0^T\int_{0}^\infty \bigg( {G_\mu(t,x)} \log {G_\mu(t,x)} -{G_\mu(t,x)}+1\bigg)  \phi\left(\int_0^{t} h(t-s)d\bar{\mu}_s\right)\mu_t(dx)dt<\infty,
$$
then  $I(\mu)<\infty$.
 
 Inversely,  if  $I(\mu)<\infty$, then, $G_\mu(t,x)=e^{\varphi_\mu(t,x)}$, and  by \eqref{general-perturbation-mean-fields-equation-sol},
\begin{align*}
J_{\mu}(\varphi)
&=\int_{0}^{T}\sum_{x=0}^{\infty}\left(\varphi(t,x) e^{\varphi_\mu(t,x)} 
- e^{\varphi(t,x)}+1\right)
\phi\left(\int_{0}^{t}h(t-u)d\bar{\mu}_u\right)f_\mu(t,x)dt.
\end{align*}

Note that $\varphi\in C^{1,b}([0,T]\times\mathbb N)$, 
while $\varphi_{\mu}$ may not be in $C^{1,b}([0,T]\times\mathbb N)$.
But we can choose a sequence $\varphi_n\in C^{1,b}([0,T]\times\mathbb N)$ such that
\begin{equation*}
\varphi_n^{(1)}\leq \varphi^{-},~~\varphi_n^{(2)}\leq \varphi^+~~\mbox{ and }~~\varphi_n:=\varphi_n^{(2)}-\varphi_n^{(1)}
\to\varphi_{\mu},~~ ~~f_\mu(t,x)dt-a.s. 
\end{equation*}
By Fatou's lemma, we obtain
\begin{align*}
I(\mu)
&=\sup_{\varphi\in C^{1,b}([0,T\times\mathbb N)}J_{\mu}(\varphi)
\\
&\geq\liminf_{n\rightarrow\infty}J_{\mu}(\varphi_{n})
\\
&\geq
\int_0^T\int_{0}^\infty \bigg( {G_\mu(t,x)} \log {G_\mu(t,x)} -{G_\mu(t,x)}+1\bigg)  \phi\left(\int_0^{t} h(t-s)d\bar{\mu}_s\right)\mu_t(dx)dt.
\end{align*}
Therefore,  \eqref{ldp-rate-funtion-representation-eq-3} holds. The proof is complete.
\end{proof}

\begin{lem}\label{ldp-rate-funtion-Property-2}
Suppose (A.1), (A.2) and (A.3) hold.  Given $\mu\in M[0,T]$ with  $I(\mu)<\infty$.  Then
there exist functions $\varphi^{(n)}\in B([0,T]\times\mathbb N)$  and  measures $\mu^{(n)}:=\mu^{\varphi^{(n)}}\in M[0,T],~n\geq 1 $ 
satisfying the equation \eqref{perturbation-mean-fields-equation} such that 
\begin{equation*}
\varphi^{(n)}(t,x)=0~ \mbox{ and }  ~\mu_t^{(n)}(\{x\})=0~ \mbox{ for any }~x\geq n+1,
\end{equation*}
and
\begin{equation} \label{ldp-rate-funtion-Property-2-eq-1}
\mu^{(n)}\to \mu   \mbox{   and   }
\lim_{n\to\infty} I\left(\mu^{(n)}\right)=I(\mu).
\end{equation}
\end{lem}

\begin{proof}
For any $n\geq 1$, we define $\mu^{(n)}\in M[0,T]$ as 
\begin{align*}
&\mu_t^{(n)}(\{x\}):=\mu_t(\{x\}),\qquad\qquad x=0,\cdots,n-1; 
\\
&\mu_t^{(n)}(\{n\}):=1-F_\mu(t,n-1);
\\
&\mu_t^{(n)}(\{x\}):=0,\qquad\qquad x\geq n+1.
\end{align*}
It follows from \eqref{general-perturbation-mean-fields-equation-sol} that
\begin{align*}
\varphi^{(n)}(t,x)
&=\varphi(t,x) I_{[0,n-1]}(x)+\log\left(\frac{\phi\left(\int_0^t 
h(t-u)d\bar\mu_u\right)}{\phi\left(\int_0^t 
h(t-u)d\bar\mu_u^{(n)}\right)}\right)I_{[0,n-1]}(x)
\\
&\qquad
+\log\left(\frac{1}{\phi\left(\int_0^t 
h(t-u)d\bar\mu_u^{(n)}\right)}\right)I_{[n,\infty)}(x)+(-\infty) I_{\{n\}}(x).
\end{align*}

We can show that
\begin{equation}\label{ConvergenceInMean}
\int_0^\infty x\mu^{(n)}_t(dx)\to \int_0^\infty x\mu_t(dx),
\end{equation}
as $n\rightarrow\infty$ uniformly in $t\in[0,T]$.
To see this, note that
\begin{equation*}
\int_0^\infty x\mu^{(n)}_t(dx)-\int_0^\infty x\mu_t(dx)
=n(1-F_{\mu}(t,n-1))-\sum_{x=n}^{\infty}x\mu_{t}(dx)
=-\frac{\sum_{x=n}^{\infty}(1-F_{\mu}(t,x))}{1-F_{\mu}(t,n-1)},
\end{equation*}
and by monotonicity of $t\mapsto F_{\mu}(t,x)$, we get
\begin{equation*}
\sup_{0\leq t\leq T}
\left|\int_0^\infty x\mu^{(n)}_t(dx)-\int_0^\infty x\mu_t(dx)\right|
\leq\frac{\sum_{x=n}^{\infty}(1-F_{\mu}(T,x))}{1-F_{\mu}(0,n-1)}
\rightarrow 0,
\end{equation*}
as $n\rightarrow\infty$.

Moreover, by noticing that, it follows from \eqref{ConvergenceInMean} that
\begin{equation*}
\phi\left( \int_0^th(t-u)d\bar{\mu}^{(n)}_u\right)\to \phi\left( \int_0^th(t-u)d\bar{\mu}_u\right).
\end{equation*}

We recall from \eqref{ldp-rate-funtion-representation-eq-0} that 
\begin{equation*} 
\int_0^t \int_0^\infty e^{\varphi^{(n)}(s,x)}\phi\left( \int_0^s h(s-u)d\bar{\mu}^{(n)}_u\right) \mu_s^{(n)}(dx)ds= \int_0^\infty x\mu^{(n)}_t(dx),
\end{equation*}
which together with \eqref{ConvergenceInMean} implies that 
\begin{align*}
&\int_0^t \int_0^\infty e^{\varphi^{(n)}(s,x)}\phi\left( \int_0^s h(s-u)d\bar{\mu}^{(n)}_u\right) \mu_s^{(n)}(dx)ds
\\
&\to \int_0^t \int_0^\infty e^{\varphi(s,x)}\phi\left( \int_0^s h(s-u)d\bar{\mu}_u\right) \mu_s(dx)ds.
\end{align*}
Moreover,
\begin{align*}
I\left(\mu^{(n)}\right)
&=\int_{0}^{T}\sum_{x=0}^{n-1}\left(\varphi(t,x)e^{\varphi(t,x)}-e^{\varphi(t,x)}+1\right)
\phi\left(\int_{0}^{t}h(t-s)d\bar{\mu}_{s}\right)\mu_{t}(\{x\})dt
\\
&\qquad
+\int_{0}^{T}\sum_{x=0}^{\infty}\log\left(\frac{\phi\left(\int_{0}^{t}h(t-s)d\bar{\mu}_{s}\right)}
{\phi\left(\int_{0}^{t}h(t-s)d\bar{\mu}^{(n)}_{s}\right)}\right)e^{\varphi(t,x)}
\phi\left(\int_{0}^{t}h(t-s)d\bar{\mu}_{s}\right)\mu_{t}(\{x\})dt
\\
&\qquad
+\int_{0}^{T}\phi\left(\int_{0}^{t}h(t-s)d\bar{\mu}^{(n)}_{s}\right)dt
-\int_{0}^{T}\phi\left(\int_{0}^{t}h(t-s)d\bar{\mu}_{s}\right)F_{t}(n-1)dt,
\end{align*}
and it follows that
\begin{equation*}
\lim_{n\to\infty}I\left(\mu^{(n)}\right)=I(\mu).
\end{equation*}
The proof is complete.
\end{proof}

\begin{lem}\label{ldp-rate-funtion-Property-3}
Suppose (A.1), (A.2) and (A.3) hold.  Given $\mu\in M[0,T]$ with  $I(\mu)<\infty$.  Assume that 
there exists a function $\varphi\in  B([0,T]\times\mathbb N)$  defined by the equation \eqref{perturbation-mean-fields-equation}, and $\varphi$ and $\mu$  satisfy the following conditions   
\begin{equation*}
\mu_t(\{x\})=0 \mbox{ and }   \varphi(t,x)=0 \mbox{ for any } x\geq m+1. 
\end{equation*}
Then
there exist measures $\mu^{(n)}\in M[0,T]$ and  functions $\varphi^{(n)}\in B([0,T]\times\mathbb N)$ satisfying  \eqref{N-dim-perturbation-Hawkes-process-varphi-exp-c}  determined by the equation \eqref{perturbation-mean-fields-equation} on $\mu^{(n)}$ and $\varphi^{(n)}$ such that 
\begin{equation} \label{ldp-rate-funtion-Property-3-eq-1}
\varphi^{(n)}(t,x)=0~ \mbox{ and }  ~\mu_t^{(n)}(\{x\})=0~ \mbox{ for any }~x\geq m+1,
\end{equation}
\begin{equation}\label{ldp-rate-funtion-Property-3-eq-2}
\inf_{t\in [0,T]}\inf _{x=0,\cdots, m}\mu_t^{(n)}(\{x\})>0,
\end{equation}
and
\begin{equation} \label{ldp-rate-funtion-Property-3-eq-3}
\mu^{(n)}\to \mu   \mbox{   and   }
\lim_{n\to\infty} I\left(\mu^{(n)}\right)=I(\mu).
\end{equation}
\end{lem}

\begin{proof}
We define
\begin{equation*}
\mu_t^{(n)}(\{x\})=\frac{\mu_t(\{x\})+\frac{1}{n}}{1+\frac{m}{n}},~~,~x=0,1,\ldots,m; 
\end{equation*}
and
\begin{equation*}
\mu_t^{(n)}(\{x\})=0,\qquad\text{for any}\quad x\geq m+1,
\end{equation*}
and thus for any $x=0,1,\ldots,m$,
\begin{equation*}
\varphi^{(n)}(t,x)=\varphi(t,x)-\log \left(\frac{\mu_t(\{x\})+\frac{1}{n}}{\mu_t(\{x\})}\right)
+\log\left(\frac{\phi\left(\int_{0}^{t}h(t-u)d\bar{\mu}_{u}\right)}{\phi\left(\int_{0}^{t}h(t-u)d\bar{\mu}_{u}^{(n)}\right)}\right),
\end{equation*}
and $\phi^{(n)}(t,x)=-\infty$ for $x=m$
and 
\begin{equation*}
\varphi^{(n)}(t,x)=
\log\left(\frac{\phi\left(\int_{0}^{t}h(t-u)d\bar{\mu}_{u}\right)}{\phi\left(\int_{0}^{t}h(t-u)d\bar{\mu}_{u}^{(n)}\right)}\right),
\end{equation*}
for any $x\geq m+1$.

Note that
\begin{equation*}
\bar{\mu}_{t}^{(n)}=\sum_{x=0}^{m}\frac{x(\mu_{t}(\{x\})+\frac{1}{n})}{1+\frac{m}{n}}
=\frac{\bar{\mu}_{t}}{1+\frac{m}{n}}
+\frac{\frac{m(m+1)}{2n}}{1+\frac{m}{n}},
\end{equation*}
and therefore,
\begin{equation*}
\sup_{0\leq t\leq T}\left|\bar{\mu}_{t}^{(n)}-\bar{\mu}_{t}\right|
\leq
\bar{\mu}_{T}\times \frac{\frac{m}{n}}{1+\frac{m}{n}}
+\frac{\frac{m(m+1)}{2n}}{1+\frac{m}{n}}\rightarrow 0,
\end{equation*}
as $n\rightarrow\infty$, where we used the fact
that $\bar{\mu}_{t}$ is increasing in $t$.

Thus $\mu_t^{(n)}(\{x\})$  and $\varphi^{(n)}(t,x)$ satisfy \eqref{ldp-rate-funtion-Property-3-eq-1}, \eqref{ldp-rate-funtion-Property-3-eq-2} and \eqref{ldp-rate-funtion-Property-3-eq-3}. The proof is complete.
\end{proof}

\begin{thm}\label{ldp-rate-funtion-Property-4}
Suppose (A.1), (A.2) and (A.3) hold.  Given $\mu\in M[0,T]$ with  $I(\mu)<\infty$.  Then
there exist a sequence of bounded functions $\varphi^{(n)}$ 
such that  measures $\mu^{(n)}:=\mu^{\varphi^{(n)}}\in M[0,T]$  
defined by the equation \eqref{perturbation-mean-fields-equation}, have the following properties:
\begin{equation} \label{ldp-rate-funtion-Property-4-eq-1}
\mu^{(n)}\to \mu   \mbox{   and   }
\lim_{n\to\infty} I\left(\mu^{(n)}\right)=I(\mu).
\end{equation}
\end{thm}

\begin{proof}
By Lemma \ref{ldp-rate-funtion-Property-2} and Lemma \ref{ldp-rate-funtion-Property-3}, we can assume that $\mu=\mu^\varphi$ defined by the equation \eqref{perturbation-mean-fields-equation}, where $\varphi\in  B([0,T]\times\mathbb N)$ and $\mu$ satisfy the following conditiions   
$$
\inf_{t\in [0,T]}\inf _{x=0,\cdots, m}\mu_t(\{x\})>0,
$$
$$
\mu_t(\{x\})=0 \mbox{ and }   \varphi(t,x)=0 \mbox{ for any } x\geq m+1. 
$$ 
Moreover, by our assumption, $\inf_{x\geq 0}\phi(x)>0$, 
and by \eqref{general-perturbation-mean-fields-equation-sol}, 
we get for every $x=0,1,\ldots,m$, $\int_{0}^{T}e^{\varphi(t,x)}dt<\infty$,
which implies that
\begin{equation*}
\int_0^T e^{\varphi^*(t)}dt<\infty.
\end{equation*}
Set $\varphi_n(t,x)=((-n)\vee \varphi(t,x))\wedge n, ~n\geq 1$. 
Then for any $x\geq m+1$,  $\varphi_n(t,x)=0$.  Set  $\mu^{(n)}=\mu^{\varphi_n}$.   Then by Theorem \ref{N-dim-perturbation-Hawkes-process-exp-Thm-2}, $\mu^{(n)}\to \mu$ in $M[0,T]$, and so
\begin{equation*}
\int_0^\infty x\mu^{(n)}_t(dx)\to \int_0^\infty x\mu_t(dx),
\end{equation*}
and
\begin{equation*}
\phi\left( \int_0^th(t-u)d\bar{\mu}^{(n)}_u\right)\to \phi\left( \int_0^th(t-u)d\bar{\mu}_u\right),
\end{equation*}
since
\begin{align*}
\int_0^th(t-u)d\bar{\mu}^{(n)}_u
&=h(0)\bar{\mu}_{t}^{(n)}-h(t)\bar{\mu}_{0}^{(n)}
+\int_{0}^{t}h'(t-u)\bar{\mu}_{u}^{(n)}du
\\
&\rightarrow
h(0)\bar{\mu}_{t}-h(t)\bar{\mu}_{0}
+\int_{0}^{t}h'(t-u)\bar{\mu}_{u}du
=\int_0^th(t-u)d\bar{\mu}_u.
\end{align*}
Since
\begin{equation*} 
\int_0^t \int_0^\infty e^{\varphi^{(n)}(s,x)}\phi\left( \int_0^s h(s-u)d\bar{\mu}^{(n)}_u\right) \mu_s^{(n)}(dx)ds= \int_0^\infty x\mu^{(n)}_t(dx),
\end{equation*}
we get
\begin{align*}
&\int_0^t \int_0^\infty e^{\varphi^{(n)}(s,x)}\phi\left( \int_0^s h(s-u)d\bar{\mu}^{(n)}_u\right) \mu_s^{(n)}(dx)ds
\\
&\to \int_0^t \int_0^\infty e^{\varphi(s,x)}\phi\left( \int_0^s h(s-u)d\bar{\mu}_u\right) \mu_s(dx)ds.
\end{align*}
By the definition of $\varphi_n$, we have that
\begin{align*}
I\left(\mu^{(n)}\right) 
&= \int_0^T\sum_{x=0}^m\left(\varphi^{(n)}(t,x)e^{\varphi^{(n)}(t,x)}-e^{\varphi^{(n)}(t,x)}+1\right) \phi\left(\int_0^{t} h(t-s)d\bar{\mu}^{(n)}_s\right)\mu_t^{(n)}(\{x\})dt
\\
&<\infty.
\end{align*}
Since $xe^x\to 0$ as $x\to -\infty$,  in order to prove $\lim_{n\to\infty}I\left(\mu^{(n)}\right)=I(\mu)$, it suffices to show that for any  $x=0,1,\cdots,m$,
\begin{equation}\label{eqn:conv}
\begin{aligned}
 &\int_0^TI_{[0,+\infty)}(\varphi(t,x)) \varphi^{(n)}(t,x)e^{\varphi^{(n)}(t,x)} \phi\left(\int_0^{t} h(t-s)d\bar{\mu}^{(n)}_s\right)\mu_t^{(n)}(\{x\})dt\\
 \to&\int_0^TI_{[0,+\infty)}(\varphi(t,x)) \varphi(t,x)e^{\varphi(t,x)}  \phi\left(\int_0^{t} h(t-s)d\bar{\mu}_s\right)\mu_t(\{x\})dt.
 \end{aligned}
\end{equation}

By Lemma \ref{ldp-rate-funtion-Property-1-lem-1},   $t\to \mu_t(\{x\})$ continuous.  By the condition  $\inf_{t\in [0,T]}\inf _{x=0,\ldots, m}\mu_t(\{x\})>0$, there exists $n_0\geq 1$  such that for any $n\geq n_0$, 
$$
0<\inf_{t\in [0,T]}\inf _{x=0,\cdots, m}\mu_t^{(n)}(\{x\})\leq \sup_{t\in [0,T]}\sup_{x=0,\cdots, m}\mu_t^{(n)}(\{x\})\leq 1.
$$
The assumptions  (A.1), (A.2) and (A.3) imply that 
$$
0<\inf_{t\in [0,T]} \phi\left(\int_0^{t} h(t-s)d\bar{\mu}_s\right)\leq \sup_{t\in [0,T]} \phi\left(\int_0^{t} h(t-s)d\bar{\mu}_s\right)<\infty. 
$$ 
Thus, there exist $n_1\geq n_0$ and positive constant $C$ such that for any $n\geq n_1$, 
$$
0<\inf_{t\in [0,T]} \phi\left(\int_0^{t} h(t-s)d\bar{\mu}^{(n)}_s\right)\leq \sup_{t\in [0,T]} \phi\left(\int_0^{t} h(t-s)d\bar{\mu}^{(n)}_s\right)\leq C.
$$
Note that under $I(\mu)<\infty$,
\begin{equation*}
\int_0^T\int_{0}^\infty  |\varphi(t,x)| e^{\varphi(t,x)}\phi\left( \int_0^th(t-u)d\bar{\mu}_u\right)\mu_t(dx)dt<\infty.
\end{equation*}
Therefore,  for $x=0,\cdots, m$,
$$
\int_0^T I_{[0,+\infty)}(\varphi(t,x)) \varphi(t,x)e^{\varphi(t,x)}dt<\infty.
$$
Now, by 
$
I_{[0,+\infty)}(\varphi(t,x))\varphi^{(n)}(t,x)e^{\varphi^{(n)}(t,x)} 
\leq I_{[0,+\infty)}(\varphi(t,x))\varphi(t,x)e^{\varphi(t,x)},
$
we have that for any $n\geq n_1$,  $x=0,\cdots, m$,
$$
\begin{aligned}
&I_{[0,+\infty)}(\varphi(t,x)) \varphi^{(n)}(t,x)e^{\varphi^{(n)}(t,x)} \phi\left(\int_0^{t} h(t-s)d\bar{\mu}^{(n)}_s\right)\mu_t^{(n)}(\{x\})\\
\leq& CI_{[0,+\infty)}(\varphi(t,x)) \varphi(t,x)e^{\varphi(t,x)}. 
\end{aligned}
$$
By the dominated convergence theorem,   \eqref{eqn:conv} holds. 
The proof is complete.
\end{proof}

%%%%%%%%%%%%%%%%%%%%%%%%%%%%%%%%%%%%%%%%%%%%%%%%%%%%%%%%%%
\section{Large deviations for the mean-field process}\label{LDPMeanFieldsSection}

In this section, we are finally ready to provide the proof of the large deviation principle 
in Theorem~\ref{main-LDP-thm-1}. Specifically, we provide the proofs
of the upper bound \eqref{main-LDP-thm-1-ub} in Section~\ref{sec:upper:bound}
and the lower bound \eqref{main-LDP-thm-1-lb} in Section~\ref{sec:lower:bound}.

%%%%%%%%%%%%%%%%%%%%%%%%%%%%%%%%
\subsection{The upper bound}\label{sec:upper:bound}

\begin{thm}\label{Ldp-up-thm-1} 
Suppose (A.1), (A.2) and (A.3) hold. Then for any compact subset $C\subset M[0,T]$,
\begin{equation}\label{Ldp-up-thm-1-eq-1}
\limsup_{N\to \infty}\frac{1}{N}\log \mathbb P\left(L^N \in C\right)\leq -\inf_{\mu\in C}I(\mu).
\end{equation}
\end{thm}

\begin{proof}
For any   
$\varphi\in C^{1,lip}([0,T]\times\mathbb N  )$,   by It\^o's formula
 \begin{equation*} 
\begin{aligned}
M_t^N:=&\left\<L^N_t,\varphi(s)\right\>-\left\langle L_{0}^{N},\varphi(0)\right\rangle
-\int_0^{t}\left\<L_s^N, \partial_s\varphi(s)+\nabla\varphi(s)\right\>
\phi\left(\int_0^{s-} 
h(s-u)d \overline{Z}^{N}_u\right)  ds\\
=&\frac{1}{N}\sum_{i=1}^N\int_0^{t}\int_{0}^{\infty}\nabla\varphi\left(s,Z_{s}^{N,i}\right)I_{z\leq\phi\left(\int_0^{s-} 
h(s-u)d \overline{Z}^{N}_u\right)}({\pi}^{i}(ds\,dz)-dsdz), 
\end{aligned}
\end{equation*}
is a martingale.
We consider the exponential martingale $\mathcal E_t^{N,\varphi}$ defined by 
\begin{equation*}  
\begin{aligned}
\mathcal E_t^{N,\varphi}
&:=\exp\bigg\{\sum_{i=1}^N\int_0^{t}\int_{0}^{\infty}\nabla\varphi\left(s,Z_{s}^{N,i}\right)I_{z\leq\phi\left(\int_0^{s-} 
h(s-u)d \overline{Z}^{N}_u\right)}({\pi}^{i}(ds\,dz)-dsdz) \\
&\quad \quad-\sum_{i=1}^N  \int_0^t   \bigg(e^{\nabla\varphi(s,Z^{N,i}_{s-})}-1-\nabla\varphi\left(s,Z^{N,i}_{s-}\right)\bigg) \phi \left( \int_0^{s-}h(s-u)d\overline{Z}_u^{N}\right)ds \bigg\}\\
&=\exp\bigg\{N\bigg(\left\<L^N_t,\varphi(s)\right\>-\left\langle L_{0}^{N},\varphi(0)\right\rangle-\int_0^{t}\left\<L_s^N, \partial_s\varphi(s)\right\>
\phi\left(\int_0^{s-} h(s-u)d \overline{Z}^{N}_u\right)  ds\\
&\quad \quad-\int_0^{t}\left\<L_s^N, e^{\nabla\varphi(s)}-1\right\> \phi\left(\int_0^{s-} h(s-u)d \overline{Z}^{N}_u\right)  ds\bigg)\bigg\}.
\end{aligned}
\end{equation*}
 Then for any measurable subset $G\subset M[0,T]$,   
 $$
\begin{aligned}
\limsup_{N\to \infty}\frac{1}{N}\log \mathbb P\left(L^N\in G\right)
=&\limsup_{N\to\infty}\frac{1}{N}\log \mathbb E\left(\left(\mathcal E_T^{N,\varphi}\right)^{-1}\mathcal E_T^{N,\varphi}I_{\{\mu^N\in G\}}\right)\\
\leq & -\inf_{\mu\in G}J_\mu(\varphi)+\limsup_{N\to \infty}\frac{1}{N}
\log \mathbb  E\left(\mathcal E_T^{N,\varphi}I_{\{\mu^N\in G\}}\right)\\
\leq & -\inf_{\mu\in G}J_\mu(\varphi).
\end{aligned}
$$
Thus,
$$
\begin{aligned}
\limsup_{N\to \infty}\frac{1}{N}\log \mathbb P\left(L^N\in G\right)
\leq & -\sup_{\varphi\in C^{1,lip}([0,T]\times\mathbb N  )}\inf_{\mu\in G}J_\mu(\varphi).
\end{aligned}
$$
Now we take measurable sets $G_j,~j=1,\cdots,k$ such that
$C\subset \bigcup_{j=1}^k G_j$, then
\begin{equation}\label{comp-uld-2}
\limsup_{n\rightarrow\infty}\frac{1}{n}\log\mathbb P\left(L^N\in G\right)
\leq-\min_{1\leq j \leq k}\sup_{\varphi\in C^{1,b}([0,T]\times\mathbb N  )} \inf_{\mu\in
G_j}J_\mu(\varphi).
\end{equation}
Set
\begin{equation*}
l:=\inf_{\mu\in C}\sup_{\varphi\in\mathcal C^{1,b}([0,T]\times\mathbb N  )}J_\mu(\varphi)\equiv - \inf_{\mu\in C}I(\mu).
\end{equation*}
For each $\mu\in C$, choose $\varphi_{\mu}\in\mathcal \mathcal C^{1,lip}([0,T]\times\mathbb N  )$ such that
$$
J_\mu(\varphi_\mu) \geq
l-\frac{\varepsilon}{2},
$$
then choose a neighborhood $N_{\mu}$ of $\mu$ such that for any
$\nu\in N_{\mu}$,
$$
J_\nu(\varphi_{\mu}) \geq
l-\varepsilon.
$$
Since $\{N_{\mu};  \mu\in C\}$ is an open covering of $C$, there
exist $\mu_1, \mu_2, \ldots, \mu_k\in C$ such that
$C\subset\bigcup_{j=1}^k N_{\mu _j}$. Set $G_j=N_{\mu_ j}$, then
\begin{equation*}
\inf_{\mu\in G_j}J_{\mu_j}(\varphi_{\mu_j}) \geq l-\varepsilon,~~~~~ 
\min_{1\leq j \leq k}\sup_{\varphi\in\mathcal \mathcal C^{1,b}([0,T]\times\mathbb N  )}\inf_{\mu\in G_j} J_\mu(\varphi) \geq l-\varepsilon,
\end{equation*}
and therefore
\begin{equation*}
\sup_{G_1,\cdots,G_k, C\subset\bigcup_{j=1}^k G_j}\min_{1\leq j\leq k}\sup_{\varphi\in\mathcal \mathcal C^{1,b}([0,T]\times\mathbb N  )}\inf_{\mu\in G_j}J_\mu(\varphi) \geq l-\varepsilon.
\end{equation*}
By (\ref{comp-uld-2}) and letting $\varepsilon\to 0$, we see that
$$
\begin{aligned}
&
\limsup_{n\rightarrow\infty}\frac{1}{n}\log\mathbb P\left(L^N\in C\right)\\
\leq &- \sup_{G_1,\cdots,G_k, C\subset\bigcup_{j=1}^k
G_j}\min_{1\leq j \leq k}\sup_{\varphi\in\mathcal \mathcal C^{1,lip}([0,T]\times\mathbb N  )}\inf_{\mu\in
G_j}J_\mu(\varphi) \leq -\inf_{\mu\in C}I(\mu).
\end{aligned}
$$
This completes the proof.
\end{proof}

\begin{thm}\label{Ldp-up-thm-2}  Suppose (A.1), (A.2)  and   (A.3) hold. Then for any closed subset $F\subset M[0,T]$,
\begin{equation}\label{Ldp-up-thm-2-eq-1}
\limsup_{N\to \infty}\frac{1}{N}\log \mathbb P\left(L^N \in F\right)\leq -\inf_{\mu\in F}I(\mu).
\end{equation}
\end{thm}

\begin{proof} 
By Theorem \ref{Ldp-up-thm-1},  
we only need to establish the exponential tightness
of the sequence $L_N$ on $M[0,T]$ (see e.g. Lemma A.1 in \cite{DGWU}, or Theorem 4.14 in \cite{FK}). It is sufficient to prove that for  $\varepsilon>0$,  $\varphi \in C^{lip}({\mathbb N})$,  
\begin{equation}\label{first:upper:bound}
\limsup_{L\rightarrow\infty}\limsup_{N\to\infty}\frac{1}{N}\log\mathbb{P}\left(\sup_{t\in[0,T]}\left|\left\<L_t^N,\varphi\right\>\right|\geq L\right)
=-\infty,
\end{equation}
and
\begin{equation}\label{Ldp-up-thm-2-eq-2}
\limsup_{\delta\to 0}\limsup_{N\to\infty}\frac{1}{N}\log \mathbb P 
\left(\sup_{0\leq t-s<\delta, s,t\in[0,T]}\left|\left\<L_t^N,\varphi\right\>-\left\<L_s^N,\varphi\right\>\right|\geq
\varepsilon\right)=-\infty.
\end{equation}

Let us first prove \eqref{first:upper:bound}.
We can compute that for any $L>|\varphi(0)|$,
\begin{align*}
\mathbb{P}\left(\sup_{t\in[0,T]}\left|\left\<L_t^N,\varphi\right\>\right|\geq L\right)
&=\mathbb{P}\left(\sup_{t\in[0,T]}\left|\sum_{i=1}^{N}\varphi\left(Z_{t}^{N,i}\right)\right|\geq NL\right)
\\
&\leq
\mathbb{P}\left(\sup_{t\in[0,T]}\sum_{i=1}^{N}\left(|\varphi(0)|+\Vert\varphi\Vert_{lip}Z_{t}^{N,i}\right)\geq NL\right)
\\
&=\mathbb{P}\left(\Vert\varphi\Vert_{lip}\sum_{i=1}^{N}Z_{T}^{N,i}\geq N(L-|\varphi(0)|)\right)
\\
&=\mathbb{P}\left(N\overline{Z}_{T}\geq\frac{N(L-|\varphi(0)|)}{\Vert\varphi\Vert_{lip}}\right).
\end{align*}
It follows from (4.2) and (4.36) in \cite{GaoFZhu-b} that for any sufficiently small $\iota>0$, there exists some $C(\iota)>0$
that depends only on $\iota,\alpha,\Vert h\Vert_{L^{1}[0,T]}, \phi(0)$ and $T$ 
with $C(\iota)\rightarrow 0$ as $\iota\rightarrow 0$ such that
\begin{equation}\label{key:bound:overline:Z}
\mathbb{E}\left[e^{\iota N\overline{Z}_{T}^{N}}\right]\leq e^{C(\iota)N}.
\end{equation} 
Thus, by applying Chebychev's inequality and \eqref{key:bound:overline:Z}, we get
\begin{equation}
\mathbb{P}\left(\sup_{t\in[0,T]}\left|\left\<L_t^N,\varphi\right\>\right|\geq L\right)
\leq
\mathbb{E}\left[e^{\iota N\overline{Z}_{T}^{N}}\right]e^{-\iota\frac{N(L-|\varphi(0)|)}{\Vert\varphi\Vert_{lip}}}
\leq
e^{C(\iota)N}e^{-\iota\frac{N(L-|\varphi(0)|)}{\Vert\varphi\Vert_{lip}}},
\end{equation}
which implies \eqref{first:upper:bound}.

Next, let us prove \eqref{Ldp-up-thm-2-eq-2}. 
Note that without loss of generality we can assume $T/\delta\in\mathbb{N}$ such that
\begin{align*}
&\mathbb P 
\left(\sup_{0\leq t-s<\delta, s,t\in[0,T]}\left|\left\<L_t^N,\varphi\right\>-\left\<L_s^N,\varphi\right\>\right|\geq
\varepsilon\right)
\\
&\leq
\mathbb{P}\left(\exists j,1\leq j\leq T/\delta:\sup_{0\leq t\leq \delta}\left|\left\<L_{t+(j-1)\delta}^N,\varphi\right\>-\left\<L_{(j-1)\delta}^N,\varphi\right\>\right|>
\frac{\varepsilon}{2}\right)
\\
&\leq
\sum_{j=1}^{T/\delta}\mathbb{P}\left(\sup_{0\leq t\leq \delta}\left|\left\<L_{t+(j-1)\delta}^N,\varphi\right\>-\left\<L_{(j-1)\delta}^N,\varphi\right\>\right|>
\frac{\varepsilon}{2}\right)
\\
&\leq
\frac{T}{\delta}\cdot\sup_{0\leq s\leq T-\delta} \mathbb P 
\left(\sup_{0\leq t\leq \delta}\left|\left\<L_{t+s}^N,\varphi\right\>-\left\<L_s^N,\varphi\right\>\right|>
\frac{\varepsilon}{2}\right).
\end{align*}
Therefore, in order to prove \eqref{Ldp-up-thm-2-eq-2}, it is sufficient to prove that for any $\varepsilon>0$, 
\begin{equation}\label{Ldp-up-thm-2-eq-2-1}
\limsup_{\delta\to 0}\limsup_{N\to\infty}\frac{1}{N}\log\sup_{0\leq s\leq T-\delta}\mathbb P 
\left(\sup_{0\leq t\leq \delta}\left|\left\<L_{t+s}^N,\varphi\right\>-\left\<L_s^N,\varphi\right\>\right|>
\varepsilon\right)=-\infty.
\end{equation}

Consider the martingale:
 \begin{equation*} 
\begin{aligned}
M_t^N:=&\left\<L^N_t,\varphi\right\>-\left\langle L_{0}^{N},\varphi\right\rangle
-\int_0^{t}\left\<L_s^N, \nabla\varphi\right\>
\phi\left(\int_0^{s-} 
h(s-u)d \overline{Z}^{N}_u\right)  ds\\
=&\frac{1}{N}\sum_{i=1}^N\int_0^{t}\int_{0}^{\infty}\nabla\varphi\left(Z_{s}^{N,i}\right)I_{z\leq\phi\left(\int_0^{s-} 
h(s-u)d \overline{Z}^{N}_u\right)}({\pi}^{i}(ds\,dz)-dsdz), 
\end{aligned}
\end{equation*}
so that
\begin{equation}\label{eqn:decomposition}
\left\<L_t^N,\varphi\right\>-\left\<L_s^N,\varphi\right\>
=M_{t}^{N}-M_{s}^{N}+\int_{s}^{t}\left\<L_v^N, \nabla\varphi\right\>
\phi\left(\int_0^{v-} 
h(v-u)d \overline{Z}^{N}_u\right) dv.
\end{equation}

First, let us prove that
\begin{equation}\label{M:inequality}
\limsup_{\delta\downarrow 0}\limsup_{N\to \infty}\frac{1}{N}
\sup_{0\leq s\leq T-\delta}\log\mathbb P\left(\sup_{0\leq t\leq \delta}\left| M_{s+t}^N-M_s^N\right|>\varepsilon\right)
=-\infty.
\end{equation}
By Chebychev's inequality and  Doob's  inequality,  for any $\lambda>0$,
\begin{equation}
\mathbb P\left(\sup_{0\leq t\leq \delta}\left| M_{s+t}^N-M_s^N\right| \geq \varepsilon\right)
\leq 2e^{-\varepsilon\lambda N}\mathbb E \left[e^{\lambda N (M_{s+\delta}^N-M_s^N)}\right].
\end{equation}
On the other hand,
\begin{align*}
M_{s+\delta}^N-M_s^N
&=\left\<L_{s+\delta}^N,\varphi\right\>-\left\<L_s^N,\varphi\right\>
-\int_{s}^{s+\delta}\left\<L_v^N, \nabla\varphi\right\>
\phi\left(\int_0^{v-} 
h(v-u)d \overline{Z}^{N}_u\right) dv
\\
&\leq
\Vert\varphi\Vert_{lip}\left(\overline{Z}_{s+\delta}^{N}-\overline{Z}_{s}^{N}\right)
+\Vert\varphi\Vert_{lip}\int_{s}^{s+\delta}\phi\left(\int_0^{v-} 
h(v-u)d \overline{Z}^{N}_u\right) dv
\\
&\leq
\Vert\varphi\Vert_{lip}\left(\overline{Z}_{s+\delta}^{N}-\overline{Z}_{s}^{N}\right)
+\Vert\varphi\Vert_{lip}\delta\left(\phi(0)+\alpha\Vert h\Vert_{L^{\infty}[0,T]}\overline{Z}^{N}_{T}\right).
\end{align*}
Therefore, we conclude that
\begin{align*}
&\mathbb P\left(\sup_{0\leq t\leq \delta}\left| M_{s+t}^N-M_s^N\right| \geq \varepsilon\right)\\
&\leq 
2e^{-\varepsilon\lambda N}
\mathbb{E}\left[e^{\lambda N\Vert\varphi\Vert_{lip}\left(\overline{Z}_{s+\delta}^{N}-\overline{Z}_{s}^{N}\right)
+\lambda N\Vert\varphi\Vert_{lip}\delta\left(\phi(0)+\alpha\Vert h\Vert_{L^{\infty}[0,T]}\overline{Z}^{N}_{T}\right)}\right]
\\
&=2e^{-\varepsilon\lambda N}e^{\lambda N\Vert\varphi\Vert_{lip}\delta\phi(0)}\mathbb{E}\left[e^{\lambda N\Vert\varphi\Vert_{lip}\left(\overline{Z}_{s+\delta}^{N}-\overline{Z}_{s}^{N}\right)
+\lambda N\Vert\varphi\Vert_{lip}\delta\alpha\Vert h\Vert_{L^{\infty}[0,T]}\overline{Z}^{N}_{T}}\right]
\\
&\leq
2e^{-\varepsilon\lambda N}e^{\lambda N\Vert\varphi\Vert_{lip}\delta\phi(0)}\left(\mathbb{E}\left[e^{2\lambda N\Vert\varphi\Vert_{lip}\left(\overline{Z}_{s+\delta}^{N}-\overline{Z}_{s}^{N}\right)}\right]\right)^{1/2}
\left(\mathbb{E}\left[e^{\lambda N\Vert\varphi\Vert_{lip}\delta\alpha\Vert h\Vert_{L^{\infty}[0,T]}\overline{Z}^{N}_{T}}\right]\right)^{1/2},
\end{align*}
where we applied Cauchy-Schwarz inequality.
By applying \eqref{key:bound:overline:Z}, we have for any $\lambda>0$,
\begin{equation}
\limsup_{\delta\to 0}\limsup_{N\to\infty}\frac{1}{N}\log\mathbb{E}\left[e^{\lambda N\Vert\varphi\Vert_{lip}\delta\alpha\Vert h\Vert_{L^{\infty}[0,T]}\overline{Z}^{N}_{T}}\right]
=0.
\end{equation}
If one can show that for any $\lambda>0$,
\begin{equation}\label{to:show}
\limsup_{\delta\to 0}\limsup_{N\to\infty}\frac{1}{N}\log\sup_{0\leq s\leq T-\delta}\mathbb{E}\left[e^{2\lambda N\Vert\varphi\Vert_{lip}\left(\overline{Z}_{s+\delta}^{N}-\overline{Z}_{s}^{N}\right)}\right]=0,
\end{equation}
then, we have
 $$
\limsup_{\delta\downarrow 0}\limsup_{N\to \infty}\frac{1}{N}
\log\sup_{0\leq s\leq T-\delta}\mathbb P\left(\sup_{0\leq t\leq \delta}\left| M_{s+t}^N-M_s^N\right|>\varepsilon\right)
\leq -\lambda\varepsilon.
$$
Since it holds for every $\lambda>0$, we conclude that
\eqref{M:inequality} holds.

Next, let us prove \eqref{to:show}. 
Note that to show \eqref{to:show}, it suffices to show that
for any $q>0$,
\begin{equation}\label{to:show:q}
\limsup_{\delta\to 0}\limsup_{N\to\infty}\frac{1}{N}\log\sup_{0\leq s\leq T-\delta}\mathbb E \left[e^{Nq\left(\overline{Z}_{s+\delta}^N-\overline{Z}_s^N\right)}\right]=0.
\end{equation}
It is easy to see that $N\overline{Z}_{t}^{N}=\sum_{i=1}^{N}Z_{t}^{N,i}$
can be viewed a one-dimensional Hawkes process with 
the intensity (see e.g. \cite{GaoFZhu-b})
\begin{equation*}
\lambda_{t}^{N}:=N\phi\left(\frac{1}{N}\int_{0}^{t-}h(t-s)d\left(N\overline{Z}_{s}^{N}\right)\right).
\end{equation*}
Thus, $e^{\theta N\overline{Z}_{t}^{N}-\int_{0}^{t}(e^{\theta}-1)\lambda_{s}^{N}ds}$ is a positive local martingale
and hence a super martingale for any $\theta\in\mathbb{R}$.  
Therefore, for any $q>0$, $\delta<T$ and $0\leq s\leq T-\delta$,
\begin{align*}
\mathbb E \left[e^{q\left(N\overline{Z}_{s+\delta}^N-N\overline{Z}_s^N\right)}\right]
&=\mathbb E \left[e^{q\left(N\overline{Z}_{s+\delta}^N-N\overline{Z}_s^N\right)-\frac{1}{2}\int_{s}^{s+\delta}(e^{2q}-1)\lambda_{u}^{N}du}
e^{\frac{1}{2}\int_{s}^{s+\delta}(e^{2q}-1)\lambda_{u}^{N}du}\right]
\\
&\leq\left(\mathbb E \left[e^{2q\left(N\overline{Z}_{s+\delta}^N-N\overline{Z}_s^N\right)-\int_{s}^{s+\delta}(e^{2q}-1)\lambda_{u}^{N}du}\right]
\mathbb{E}\left[e^{\int_{s}^{s+\delta}(e^{2q}-1)\lambda_{u}^{N}du}\right]\right)^{1/2}
\\
&\leq\left(\mathbb{E}\left[e^{\int_{s}^{s+\delta}(e^{2q}-1)\lambda_{u}^{N}du}\right]\right)^{1/2}
\\
&\leq\left(\mathbb{E}\left[e^{\int_{s}^{s+\delta}(e^{2q}-1)(N\phi(0)+\alpha\int_{0}^{u-}h(u-v)d(N\overline{Z}_{v}^{N})du}\right]\right)^{1/2}
\\
&\leq\left(\mathbb{E}\left[e^{\int_{s}^{s+\delta}(e^{2q}-1)(N\phi(0)+\alpha\Vert h\Vert_{L^{\infty}[0,u]}N\overline{Z}_{u}^{N})du}\right]\right)^{1/2}
\\
&\leq\left(\mathbb{E}\left[e^{\delta(e^{2q}-1)(N\phi(0)+\alpha\Vert h\Vert_{L^{\infty}[0,T]}N\overline{Z}_{T}^{N})}\right]\right)^{1/2},
\end{align*}
where we applied Cauchy-Schwarz inequality to obtain the first inequality above.
Finally, by applying \eqref{key:bound:overline:Z}, we obtain \eqref{to:show:q}. Hence, we proved \eqref{to:show} and thus \eqref{M:inequality}.
 
Next, we will show that
\begin{equation}\label{L:inequality}
\limsup_{\delta\downarrow 0}\limsup_{N\to \infty}\frac{1}{N}
\log\sup_{0\leq s\leq T-\delta}\mathbb P\left(\sup_{0\leq t\leq \delta}\left| \int_{s}^{s+t}\left\<L_v^N, \nabla\varphi\right\>   \phi\left(\int_0^{v-} 
h(v-u)d \overline{Z}^{N}_u\right)  dv\right|>\varepsilon\right)
=-\infty.
\end{equation}
To show \eqref{L:inequality}, we can compute
that for any $\delta<T$ and $0\leq s\leq T-\delta$,
\begin{align*}
&\mathbb P\left(\sup_{0\leq t\leq \delta}\left| \int_{s}^{s+t}\left\<L_v^N, \nabla\varphi\right\>   \phi\left(\int_0^{v-} 
h(v-u)d \overline{Z}^{N}_u\right)  dv\right|>\varepsilon\right)
\\
&=
\mathbb P\left(\sup_{0\leq t\leq \delta}\left| \int_{s}^{s+t}\sum_{i=1}^{N}\nabla\varphi\left(Z_{v}^{N,i}\right)\phi\left(\int_0^{v-} 
h(v-u)d \overline{Z}^{N}_u\right)  dv\right|>N\varepsilon\right)
\\
&\leq
\mathbb P\left(\int_s^{s+\delta}N\Vert\varphi\Vert_{lip}\phi\left(\int_0^{v-} 
h(v-u)d \overline{Z}^{N}_u\right)  dv>N\varepsilon\right)
\\
&\leq
\mathbb P\left(N\Vert\varphi\Vert_{lip}\int_{s}^{s+\delta}\left(\phi(0)+\alpha\int_0^{v-} 
h(v-u)d \overline{Z}^{N}_u\right) dv>N\varepsilon\right)
\\
&\leq
\mathbb P\left(N\Vert\varphi\Vert_{lip}\int_{s}^{s+\delta}\alpha\Vert h\Vert_{L^{\infty}[0,v]}\overline{Z}^{N}_{v}dv>N\left(\varepsilon-\Vert\varphi\Vert_{lip}\phi(0)\delta\right)\right)
\\
&\leq
\mathbb P\left(N\overline{Z}^{N}_{T}>\frac{N\left(\varepsilon-\Vert\varphi\Vert_{lip}\phi(0)\delta\right)}{\delta\Vert\varphi\Vert_{lip}\alpha\Vert h\Vert_{L^{\infty}[0,T]}}\right).
\end{align*}
Finally, by applying Chebychev's inequality and \eqref{key:bound:overline:Z}, we obtain \eqref{L:inequality}.

Hence, by applying \eqref{eqn:decomposition}, \eqref{M:inequality} and \eqref{L:inequality}, 
we conclude that \eqref{Ldp-up-thm-2-eq-2-1} holds and thus \eqref{Ldp-up-thm-2-eq-2} also holds. The proof is complete.
\end{proof}

%%%%%%%%%%%%%%%%%%%%%%%%%%%%%%%%%%%%%%%%%
\subsection{The lower bound}\label{sec:lower:bound}

\begin{thm} \label{Ldp-lb-thm-1}
Suppose (A.1), (A.2)  and   (A.3) hold. Then for any   $\mu\in M[0,T]$ and open set  $O\ni \mu$, 
\begin{equation} \label{Ldp-lb-thm-1-eq-1}
\liminf_{N\to\infty}\frac{1}{N}\log \mathbb P \left(L^N\in O\right)
\geq-I(\mu).
\end{equation}
\end{thm}

\begin{proof}
If  $I(\mu)=+\infty$,   then (\ref{Ldp-lb-thm-1-eq-1})  is obvious. Next, let us assume $I(\mu)<\infty$.  
By Theorem \ref{ldp-rate-funtion-Property-2}, we can assume that there exists a bounded  $\varphi \in B([0,T]\times\mathbb N)$ such that that $\mu=\mu^\varphi$ defined by the equation \eqref{perturbation-mean-fields-equation},  and $\mu$ satisfies   
$$
\inf_{t\in [0,T]}\inf _{x=0,\cdots, m}\mu_t(\{x\})>0,
$$
and
$$
\mu_t(\{x\})=0 \mbox{ and }   \varphi(t,x)=0 \mbox{ for any } x\geq m+1. 
$$ 
 
Let  $\left\{Z^{\varphi, N,i}_t,t\in[0,T]\right\}$ be the unique solution of the SDEs \eqref{N-dim-perturbation-Hawkes-process-eq}.  Let $\left\{\tilde Z_t^\varphi, t\in[0,T]\right\}$ be the unique solution of the equation \eqref{Hawkes-perturbation-mean-eq-1}  and let $\mathcal L_t^{\varphi}(dx)$ denote the distribution of $\left\{\tilde Z_t^\varphi, t\in[0,T]\right\}$.  We recall from \eqref{N-dim-perturbation-mean-fields-eq} that
$$
L^{\varphi, N}_t=\frac{1}{N}\sum_{i=1}^N\delta_{Z^{\varphi, N,i}_t}.
$$
Then by Theorem~\ref{perturbation-mean-fields-LLN-thm-1},   $\mu=\mathcal L^{\varphi}$,  and for any   an open subset $O\ni \left(\mathcal L_t^{\varphi}\right)_{t\in[0,T]} $,
\begin{equation} \label{ldp-lb-thm-1-eq-3}
\lim_{N\to\infty}\mathbb P^{\varphi,N} \left(L^{\varphi, N}\in O\right)=1,
\end{equation}
where $\mathbb P^{\varphi,N}$ denote the probability law of $\left(Z^{\varphi,N,1}_t,\cdots,Z^{\varphi,N,N}_t\right)_{t\in [0,T]}$.
By definition, $\mathbb P^{0,N}=\mathbb{P}$.

Moreover, let $Z_t^N:=\left(Z^{N,1}_t,\cdots,Z^{N,N}_t\right)_{t\in [0,T]}$ denote the coordinate process. Then
\begin{equation*} \label{N-dim-perturbation-Hawkes-process-NR-eq}
\begin{aligned}
\log\frac{d\mathbb P^{\varphi,N}}{d\mathbb P^{0,N}}
&=\sum_{i=1}^N\int_0^T \int_0^\infty \left(- \varphi\left(s,Z^{N,i}_{s-}\right)\right)I_{\left\{z \leq e^{\varphi(s,Z^{N,i}_{s-})}\phi \left( \int_0^{s-}h(s-u)d\overline{Z}_u^{N}\right)\right\}}
(\pi^i(ds\,dz)-dsdz) \\
&\quad +\sum_{i=1}^N  \int_0^T   \bigg(e^{-\varphi(s,Z^{N,i}_{s-})}-1+\varphi\left(s,Z^{N,i}_{s-}\right)\bigg)e^{\varphi\left(s,Z^{N,i}_{s-}\right)}  \phi \left( \int_0^{s-}h(s-u)d\overline{Z}_u^{N}\right)ds \\
&=N\left(-M_T^{N}+\int_0^T\left\<L_N^\varphi(t),-e^{\varphi(t)}+1+\varphi(t) e^{\varphi(t)}\right\>dt\right),
\end{aligned}
\end{equation*}
where
$$
M_t^{N}:=\frac{1}{N}\sum_{i=1}^N\int_0^t \int_0^\infty  \varphi\left(s,Z^{N,i}_{s-}\right)I_{\left\{z \leq e^{\varphi(s,Z^{N,i}_{s-})}\phi \left( \int_0^{s-}h(s-u)d\overline{Z}_u^{N}\right)\right\}}
(\pi^i(ds\,dz)-dsdz).
$$
Since
\begin{equation*}
\left\<M^{N}\right\>_T=\frac{1}{N^2}\sum_{i=1}^N\int_0^T  \bigg(\varphi\left(s,Z^{N,i}_{s-}\right)\bigg)^2 e^{\varphi\left(s,Z^{N,i}_{s-}\right)}\phi \left( \int_0^{s-}h(s-u)d\overline{Z}_u^{N}\right) ds,
\end{equation*}
we have
\begin{equation*}
\sup_{t\in[0,T]}\left|M_t^{N}\right|\to 0\qquad\text{in probability}. 
\end{equation*}

Now, by
$$
\begin{aligned}
&~~\frac{1}{N}\log \mathbb P\left(L^N\in O\right)\\
&=\frac{1}{N}\log \frac{1}{\mathbb P^{\varphi,N }(L^{\varphi, N}\in O)}
\mathbb E^{\varphi, N}\left(I_{\{L^N\in O\}} \frac{d\mathbb P^{0, N}}{d\mathbb P^{\varphi, N}}\right)
+\frac{1}{N}\log \mathbb P^{\varphi, N}\left(L^{\varphi, N}\in O\right)\\
&{\geq} \frac{1}{N} 
\mathbb E^{\varphi, N}\left(I_{\{L^N\in O\}} \log\frac{d\mathbb P^{0, N}}{d\mathbb P^{\varphi, N}}\right)
+\frac{1}{N}\log \mathbb P^{\varphi, N}\left(L^{\varphi, N}\in O\right),
\end{aligned}
$$
where $\mathbb E^{\varphi,N}$ is the associated expectation of $\mathbb P^{\varphi,N}$,
we obtain
 $$
\begin{aligned}
&\liminf_{N\to\infty}\frac{1}{N}\log \mathbb P\left(L^N\in O\right)\\
\geq&-\limsup_{N\to\infty}\mathbb E^{\varphi, N}\left(I_{\{L^{\varphi, N}\in O\}}\int_0^T\left\langle L_N^\varphi(t),-e^{\varphi(t)}+1+\varphi(t) e^{\varphi(t)}\right\rangle dt\right)\\
=&-I(\mu).
\end{aligned}
$$
This completes the proof.
\end{proof}

Thus, we have the following result.

 \begin{thm} \label{Ldp-lb-thm-2}
Suppose (A.1), (A.2)  and   (A.3) hold. Then for any   open set  $O$ in $ M[0,T]$, 
\begin{equation} \label{Ldp-lb-thm-2-eq-1}
\liminf_{N\to\infty}\frac{1}{N}\log \mathbb P \left(L^N\in O\right)
\geq-\inf_{\mu\in O}I(\mu).
\end{equation}
\end{thm}

%%%%%%%%%%%%%%%%%%%%%%%%%%%%%%%%%%%%%%%%%%%%%%%%%%
\section{Further Discussions}\label{sec:discussions}

In this section, we provide further discussions. In particular, 
we will show that our result recovers the large deviation principle
for the mean process of the Hawkes processes obtained in \cite{GaoFZhu-b}.
We first start with the following observation.

\begin{rmk}
Since the mean-field limit $\mu$ is the law of an inhomogeneous Poisson process
with compensator $m(t)$ and intensity $\phi\left(\int_{0}^{t}h(t-s)dm(s)\right)$ (see \cite{Delattre}), 
for any $x\in\mathbb{N}$,
\begin{equation*}
F_{\mu}(t,x)=\sum_{k=0}^{x}\frac{e^{-m(t)}}{k!}(m(t))^{k}.
\end{equation*}
By differentiating with respect to $t$ in the above equation, we get
\begin{equation*}
\partial_{t}F_{\mu}(t,x)=-m'(t)F_{\mu}(t,x)+m'(t)F_{\mu}(t,x-1)
=-f_{\mu}(t,x),
\end{equation*}
which implies from \eqref{general-perturbation-mean-fields-equation-sol} that $\varphi_{\mu}\equiv 0$ 
and by \eqref{ldp-rate-funtion-representation-eq-1}, we have $I(\mu)=0$.
\end{rmk}

As a consequence of our main result Theorem~\ref{main-LDP-thm-1}, we obtain the following corollary.

\begin{cor}\label{mean-proce-LDPThm}
$\mathbb{P}\left(\overline{Z}_{t}^{N}\in\cdot\right)$ satisfies a large deviation
principle on $D[0,T]$ equipped with Skorokhod topology
with the speed $N$ and the rate function
\begin{equation}\label{IEqn}
I(\eta):=\int_{0}^{T}\ell\left(\eta'(t);\phi\left(\int_{0}^{t}h(t-s)d\eta(s)\right)\right)dt,
\end{equation}
if $\eta\in\mathcal{AC}_{0}^{+}[0,T]$ and $+\infty$ otherwise, where
\begin{equation}\label{ellEqn}
\ell(x;y):=x\log\left(\frac{x}{y}\right)-x+y,
\end{equation}
and $\mathcal{AC}_{0}^{+}[0,T]$ is the space of non-decreasing functions 
$f:[0,T]\rightarrow\mathbb{R}$ that are absolutely continuous
with $f(0)=0$.
\end{cor}

\begin{proof}
Note that the map $\mu_{t}\mapsto\int_{0}^{\infty}x\mu_{t}(dx)$ is continuous. 
By the contraction principle (see e.g. \cite{Dembo}) and our Theorem~\ref{main-LDP-thm-1}, 
$\mathbb{P}\left(\overline{Z}_{t}^{N}\in\cdot\right)$ satisfies a large deviation
principle on $D[0,T]$ equipped with Skorokhod topology
with the speed $N$ and the rate function
\begin{align}\label{minimum}
I(\eta)&=\inf_{\mu:\int_{0}^{\infty}x\mu_{t}(dx)=\eta(t)}
\int_0^T\int_{0}^\infty \left(\varphi_\mu(t,x)e^{\varphi_\mu(t,x)}-e^{\varphi_\mu(t,x)}+1\right)
\\
&\quad\quad\quad\cdot\phi\left(\int_0^{t} h(t-s)d\bar{\mu}_s\right)\mu_t(dx)dt.
\nonumber
\end{align}
Notice that, for every $T$,
\begin{equation}\label{diffT}
\int_0^T\int_{0}^\infty e^{\varphi_\mu(t,x)}\phi\left(\int_0^{t} h(t-s)d\eta(s)\right)\mu_t(dx)dt
=\eta(T),
\end{equation}
and
\begin{equation}
\int_0^T\int_{0}^\infty \phi\left(\int_0^{t} h(t-s)d\bar{\mu}_s\right)\mu_t(dx)dt
=\int_0^T \phi\left(\int_0^{t} h(t-s)d\eta(s)\right)dt.
\end{equation}
Moreover, since the map $x\mapsto x\log x$ is convex, 
by Jensen's inequality
\begin{equation}
\int_{0}^{\infty}\varphi_{\mu}(t,x)e^{\varphi_{\mu}(t,x)}\mu_{t}(dx)
\geq
\int_{0}^{\infty}e^{\varphi_{\mu}(t,x)}\mu_{t}(dx)\log\int_{0}^{\infty}e^{\varphi_{\mu}(t,x)}\mu_{t}(dx).
\end{equation}
Finally, notice that by differentiating \eqref{diffT} w.r.t. $T$ and set $T=t$, we get
\begin{equation}
\eta'(t)=\phi\left(\int_0^{t} h(t-s)d\eta(s)\right)\int_{0}^\infty e^{\varphi_\mu(t,x)}\mu_t(dx).
\end{equation}
The minimum in \eqref{minimum} is achieved at
\begin{equation}
\varphi_\mu(t,x)=\log\left(\frac{\eta'(t)}{\phi\left(\int_0^{t} h(t-s)d\eta(s)\right)}\right).
\end{equation}
The conclusion follows.
\end{proof}

\begin{rmk}
Suppose (A.1), (A.2) and (A.3) hold.
Then, it is proved in \cite{GaoFZhu-b} that
$\mathbb{P}\left(\overline{Z}_{t}^{N}\in\cdot\right)$ satisfies a large deviation
principle on $D[0,T]$ equipped with Skorokhod topology
with the speed $N$ and the rate function defined in \eqref{IEqn}.
Hence, Corollary~\ref{IEqn} recovers the large deviations result in \cite{GaoFZhu-b}.
\end{rmk}

\section*{Acknowledgements}
Fuqing Gao acknowledges support from NSFC Grants 11971361 and 11731012.
Lingjiong Zhu is grateful to the support from NSF Grants DMS-1613164, DMS-2053454, 
DMS-2208303, and a Simons Foundation Collaboration Grant.

%%%%%%%%%%%%%%%%%%%%%%%%%%%%%%%%%%%%%%%%%%%%%%%%%%%%%%%%%%%%%%%%


\begin{thebibliography}{99}

\bibitem{Agathe}
Agathe-Nerine, Z. (2022).
Multivariate Hawkes processes on inhomogeneous random graphs.
\textit{Stochastic Processes and their Applications}.
\textbf{152}, 86-148.

\bibitem{Bacry}
Bacry, E, Delattre, S., Hoffmann, M. and  Muzy, J. F. (2013).
Scaling limits for Hawkes processes and application to financial statistics.
\textit{Stochastic Processes and their Applications}.
\textbf{123}, 2475-2499.

\bibitem{Bacry2016}
Bacry, E., Ga\"{i}ffas, S., Mastromatteo, I. and Muzy, J. (2016).
Mean-field inference of Hawkes point processes.
\textit{Journal of Physics A: Mathematical and Theoretical}.
\textbf{49}, 174006.

\bibitem{Billingsley}
Billingsley, P. (1999). 
\textit{Convergence of Probability Measures}, 
2nd edition. Wiley-Interscience, New York.

\bibitem{Blanchet2025}
Blanchet, J., Laeven, R.J.A, Wang, X. and Zwart, B. (2025).
Sample path large deviations for multivariate heavy-tailed Hawkes processes and related L\'{e}vy processes.
\textit{arXiv:2504.01119}.

\bibitem{Bordenave}
Bordenave, C. and Torrisi, G. L. (2007).
Large deviations of Poisson cluster processes. 
\textit{Stochastic Models}. 
\textbf{23}, 593-625.

\bibitem{Borovykh}
Borovykh, A., Pascucci, A. and La Rovere, S. (2018).
Systemic risk in a mean-field model of interbank lending with self-exciting shocks.
\textit{IISE Transaction}.
\textbf{50}, 806-819.

\bibitem{Bremaud}
Br\'{e}maud, P. and Massouli\'{e}, L. (1996).
Stability of nonlinear Hawkes processes. 
\textit{Annals of Probability} 
\textbf{24}, 1563-1588.
  
\bibitem{Chevallier}
Chevallier, J. (2017).
Mean-field limit of generalized Hawkes processes.
\textit{Stochastic Processes and their Applications}.
\textbf{127}, 3870-3912.

\bibitem{ChevallierII}
Chevallier, J. (2017).
Fluctuations for mean-field interacting age-dependent Hawkes processes.
\textit{Electronic Journal of Probability}.
\textbf{22}, No.42, 49pp.

\bibitem{ChevallierIII}
Chevallier, J. (2018).
Stimulus sensitivity of a spiking neural network model. 
\textit{Journal of Statistical Physics}.
\textbf{170}, 800-808.

\bibitem{CDLO}
Chevallier, J., Duarte, A., L\"{o}cherbach, E. and G. Ost. (2019).
Mean field limits for nonlinear spatially extended Hawkes processes with exponential memory kernels.
\textit{Stochastic Processes and their Applications}.
\textbf{129}, 1-27.

\bibitem{Chevallier2021}
Chevallier, J., Melnykova, A. and Tubikanec, I. (2021).
Diffusion approximation of multi-class Hawkes processes: Theoretical and numerical analysis.
\textit{Advances in Applied Probability}.
\textbf{53}, 716-756.

\bibitem{Costa}
Costa, M., Graham, C., Marsalle, L. and Tran, V. C. (2020).
Renewal in Hawkes processes with self-excitation and inhibition.  
\textit{Advances in Applied Probability}.  \textbf{52},  879--915. 

\bibitem{DF}
Delattre, S. and  Fournier, N.(2016).
Statistical inference versus mean field limit for Hawkes processes.
\textit{Electronic Journal of Statistics}.
\textbf{10}, 1223-1295.

\bibitem{Delattre}
Delattre, S., Fournier, N. and Hoffmann, M. (2016)
Hawkes processes on large networks.
\textit{Annals of Applied Probability}.
\textbf{26}, 216-261.

\bibitem{Dembo}
Dembo, A. and  Zeitouni, O.
\textit{Large Deviations Techniques and Applications}. 
2nd Edition, Springer, New York, 1998.

\bibitem{DL}
Ditlevsen, S. and E. L\"{o}cherbach. (2017).
Multi-class oscillating systems of interacting neurons.
\textit{Stochastic Processes and their Applications}.
\textbf{127}, 1840-1869.

\bibitem{DGWU}
Djellout, H., Guillin, A. and Wu, L. M. (1999).
Large and moderate deviations for estimators of quadratic variational processes of diffusion. 
\textit{Statistical Inference for Stochastic Processes}. 
\textbf{2}, 195-225.

\bibitem{Duarte2025}
Duarte, A., Laxa, K., L\"{o}cherbach, E. and Loukianova, D. (2025).
Nonparametric estimation of the jump rate in mean field interacting systems of neurons.
\textit{arXiv:2506.24065}.

\bibitem{Duval}
Duval, C., Lu\c{c}on, E. and Pouzat, C. (2022).
Interacting Hawkes processes with multiplicative inhibition.
\textit{Stochastic Processes and their Applications}.
\textbf{148}, 180-226.

\bibitem{Euch2018}
El Euch, O., Fukasawa, M. and Rosenbaum, M. (2018).
The microstructral foundations of leverage effect and rough volatility.
\textit{Finance and Stochastics}.
\textbf{22}, 241-280.

\bibitem{Erny}
Erny, X., L\"{o}cherbach, E. and Loukianova, D. (2022).
Mean field limits for interacting Hawkes processes in a diffusive regime.
\textit{Bernoulli}.  \textbf{28}, 125--149.

\bibitem{FK}
Feng, J. and Kurtz, T. G.
\textit{Large Deviations for Stochastic Processes}.
American Mathematical Society, 2006.

\bibitem{Flint}
Flint, I., Privault, N. and Torrisi, G. L. (2019).
Functional inequalities for marked point processes. 
\textit{Electronic Journal of Probability}.  
\textbf{24}, 1--40.

\bibitem{GaoFZhu}
Gao, F. Q. and Zhu, L. (2021).
Precise deviations for Hawkes processes.
\textit{Bernoulli}. \textbf{27}, 221--248. 
 
\bibitem{GaoFZhu-b}
Gao, F. Q. and Zhu, L. (2018).
Some asymptotic results for nonlinear Hawkes processes.
\textit{Stochastic Processes and their Applications}.
\textbf{128}, 4051-4077.

\bibitem{GZ}
Gao, X. and Zhu, L. (2018).
Limit theorems for linear Markovian Hawkes processes with large initial intensity. 
\textit{Stochastic Processes and their Applications}.
\textbf{128}, 3807-3839.

\bibitem{GZ2}
Gao, X. and Zhu, L. (2018). 
Large deviations and applications for Markovian Hawkes processes with a large initial intensity. 
\textit{Bernoulli}.
\textbf{24}, 2875-2905.

\bibitem{GZ3}
Gao, X. and Zhu, L. (2018).
A functional central limit theorem for stationary Hawkes processes and its application to infinite-server queues.
\textit{Queueing Systems}.
\textbf{90}, 161-206.

\bibitem{Graham}
Graham, C. (2021).
Regenerative properties of the linear Hawkes process with unbounded memory.
\textit{Annals of Applied Probability}. 
\textbf{31}, 2844--2863.
 
\bibitem{GrigorescuCPAM}
Grigorescu, I. (2007).   
Large Deviations for a Catalytic Fleming-Viot Branching System.   
\textit{Communications on Pure and Applied Mathematics}.  
\textbf{LX} , 1056--1080.

\bibitem{GPV-CPAM}
Guo, M.Z. Papanicolaou, G.C. and Varadhan, S.R.S. (1988).  
Nonlinear diffusion limit for a system with nearest neighbor interactions. 
\textit{Communications in Mathematical Physics}. 
\textbf{118}, 31-59.
 
\bibitem{Hawkes}
Hawkes, A. G. (1971).
Spectra of some self-exciting and mutually exciting point processes.
\textit{Biometrika}. 
\textbf{58}, 83-90.

\bibitem{HawkesII}
Hawkes, A. G. and Oakes, D. (1974).
A cluster process representation of a self-exciting process.
\textit{Journal of Applied Probability}. 
\textbf{11}, 493-503.

\bibitem{Heesen}
Heesen, S. and Stannat, W. (2021).
Fluctuation limits for mean-field interacting nonlinear Hawkes processes.
\textit{Stochastic Processes and their Applications}.
\textbf{139}, 280-297.

\bibitem{HorstXu2021}
Horst, U. and Xu, W. (2021).
Functional limit theorems for marked Hawkes point measures.
\textit{Stochastic Processes and their Applications}.
\textbf{134}, 94-131.

\bibitem{HorstXu2022}
Horst, U. and Xu, W. (2022).
The microstructure of stochastic volatility models with self-exciting jump dynamics.
\textit{Annals of Applied Probability}.
\textbf{32}, 4568-4610.

\bibitem{HorstXu2025}
Horst, U. and Xu, W. (2026).
Functional limit theorems for Hawkes processes.
\textit{Probability Theory and Related Fields}.
\textbf{194}, 917-996.

\bibitem{Horst2023}
Horst, U., Xu, W. and Zhang, R. (2023).
Convergence of heavy-tailed Hawkes processes and the microstructure of rough volatility.
\textit{arXiv:2312.08784}.

\bibitem{JacodShiryaev}
Jacod, J and  Albert N. Shiryaev, A. N. (2003).
{\sl Limit Theorems for Stochastic Processes.}  Second Edition.
Grundlehren der mathematischen Wissenschaften  288. Springer.

\bibitem{Jaisson}
Jaisson, T. and Rosenbaum, M. (2015).
Limit theorems for nearly unstable Hawkes processes.
\textit{Annals of Applied Probability}.
\textbf{25}, 600-631.

\bibitem{JaissonII}
Jaisson, T. and Rosenbaum, M. (2016).
Rough fractional diffusions as scaling limits of nearly unstable heavy tailed Hawkes processes.
\textit{Annals of Applied Probability}.
\textbf{26}, 2860-2882.

\bibitem{KarabashZhu2015}
Karabash, D. and Zhu, L. (2015).
Limit theorems for marked Hawkes processes with application to a risk model.
\textit{Stochastic Models}.
\textbf{31}, 433-451.

\bibitem{Karim2025}
Karim, R.S., Laeven, R.J.A. and Mandjes, M. (2025).
Compound multivariate Hawkes processes: Large deviations and rare event simulation.
\textit{Bernoulli}.
\textbf{31}, 3113-3138.
 
\bibitem{KipnisOllaVaradhan}
Kipnis, C.,  Olla, S. and Varadhan, S. R. S. (1989). 
Hydrodynamics and large deviations for simple exclusion processes.    
\textit{Communications on Pure and Applied Mathematics}.   
\textbf{42}, 115-137.

\bibitem{LiPang2022SPA}
Li, B. and Pang, G. (2022).
Functional limit theorems for nonstationary marked Hawkes processes in the high intensity regime.
\textit{Stochastic Processes and their Applications}.
\textbf{143}, 285-339.

\bibitem{Locherbach2017}
L\"{o}cherbach, E. (2017).
Spking neurons: interacting Hawkes processes, mean field limits and oscillations.
\textit{ESAIM: Proceedings and Surveys}.
\textbf{60}, 90-103.

\bibitem{Locherbach}
L\"{o}cherbach, E. (2019).
Large deviations for cascades of diffusions arising in oscillating systems of interacting Hawkes processes.
\textit{Journal of Theoretical Probability}.
\textbf{32}, 131-162.

\bibitem{Lucon2025}
Lu\c{c}on, E. and Poquet, C. (2025).
Neural field equations and Hawkes processes: Long-term stability of traveling wave profiles in the neutral case.
\textit{arXiv:2507.19236}.

\bibitem{Pfaffelhuber}
Pfaffelhuber, P., Rotter, S. and Stiefel, J. (2022).
Mean-field limits for non-linear Hawkes processes with excitation and inhibition.
\textit{Stochastic Processes and their Applications}.
\textbf{153}, 57-78.

\bibitem{Raad}
Raad, M. B., Ditlevsen, S. and L\"{o}cherbach, E. (2020).
Stability and mean-field limits of age dependent Hawkes processes.
\textit{Annales de l'Institut Henri Poincar\'{e}-Probabilit\'{e}s et Statistiques}.  
\textbf{56}, 1958--1990.

\bibitem{RevuzYor}
Revuz, D. and  Yor, M.
\textit{Continuous Martingales and Brownian Motion}.
Springer, 3rd Edition, 1998.

\bibitem{Rosenbaum2021}
Rosenbaum, M. and Tomas, M. (2021).
From microscopic price dynamics to multidimensional rough volatility models.
\textit{Annals of Applied Probability}.
\textbf{53}, 425-462.

\bibitem{Schmutz}
Schmutz, V. (2022). 
Mean-field limit of age and leaky memory dependent Hawkes processes.
\textit{Stochastic Processes and their Applications}.
\textbf{149}, 39-59.

\bibitem{SH}
Sokol, A. and Hansen, N. R. (2015).
Exponential martingales and changes of measure for counting processes.
\textit{Stochastic Analysis and Applications}.
\textbf{33}, 823-843.

\bibitem{Stiefel2023}
Stiefel, J. (2023).
Mean-field limits for non-linear Hawkes processes with
inhibition on a Erd\H{o}s-R\'{e}nyi-graph.
\textit{ALEA, Latin American Journal of Probability and Mathematical Statistics}.
\textbf{20}, 1459-1481.

\bibitem{Szymanski-Xu-2025}
Szymanski, G. and Xu, W. (2025).
Mean-field limits for nearly unstable Hawkes processes.
\textit{arXiv:2501.11648}.

\bibitem{TorrisiI}
Torrisi, G. L. (2016).
Gaussian approximation of nonlinear Hawkes processes.
\textit{Annals of Applied Probability}.
\textbf{26}, 2106-2140.

\bibitem{TorrisiII}
Torrisi, G. L. (2017).
Poisson approximation of point processes with stochastic intensity,
and application to nonlinear Hawkes processes.
\textit{Annales de l'Institut Henri Poincar\'{e}-Probabilit\'{e}s et Statistiques}.
\textbf{53}, 679-700.

\bibitem{VaradhanII} 
Varadhan, S. R. S. 
\textit{Large Deviations and Applications}, SIAM, Philadelphia, 1984.

\bibitem{Villani2003}
Villani, C. Topics in Optimal Transportation.
Volume 58 of Graduate Studies in Mathematics.
American Mathematical Society. Providence, RI, 2003.

\bibitem{Xu2024}
Xu, W. (2024).
Diffusion approximations for marked self-excited systems with applications to general branching processes.
\textit{Annals of Applied Probability}.
\textbf{34}, 2733-2798.

\bibitem{Xu2024arXiv}
Xu, W. (2024).
Scaling limit theorems for multivariate Hawkes processes and stochastic Volterra equations with measure kernel.
\textit{arXiv:2412.14459}.

\bibitem{XuVolterraI}
Xu, W. (2024).
Stochastic VolterraI equations for local times of spectrally positive L\'{e}vy processes with Gaussian components.
\textit{arXiv:2411.15485}.

\bibitem{XuVolterraII}
Xu, W. (2024).
Stochastic VolterraI equations for the local times of spectrally positive stable processes.
\textit{Annals of Applied Probability}.
\textbf{34}, 2733-2798.

\bibitem{ZhuThesis} 
Zhu, L. (2013).
\textit{Nonlinear Hawkes Processes}. PhD thesis, New York University.

\bibitem{ZhuCLT} 
Zhu, L. (2013).
Central limit theorem for nonlinear Hawkes processes.
\textit{Journal of Applied Probability}.
\textbf{50} 760-771.
 
\bibitem{ZhuII} 
Zhu, L. (2014).
Process-level large deviations for nonlinear Hawkes point processes.
\textit{Annales de l'Institut Henri Poincar\'{e}-Probabilit\'{e}s et Statistiques}.
\textbf{50}, 845-871.

\bibitem{ZhuI} 
Zhu, L. (2015).
Large deviations for Markovian nonlinear Hawkes Processes.
\textit{Annals of Applied Probability}.
\textbf{25}, 548-581.
\end{thebibliography}
\end{document}